\documentclass[twocolumn]{autart}    % Enable this line and disable the 
\usepackage{graphicx}          % Include this line if your 
\usepackage{amsmath}
\usepackage{amssymb}
\usepackage[utf8]{inputenc}
\usepackage{float}
\usepackage{caption,booktabs}
\newtheorem{lemma}{Lemma}
\newtheorem{coro}{Corollary}

\newenvironment{proof}{\noindent {\em Proof.}}{\hfill \hspace*{1pt} \hfill $\square$}

\newtheorem{nnremark}{\bf Remark}
\newenvironment{remark}{\begin{nnremark} \rm }{\hfill \hspace*{1pt}\hfill $\lrcorner$\end{nnremark}}

\begin{document}
	
	\begin{frontmatter}
		
		\title{State feedback control law design \\for an age-dependent SIR model} % Title, preferably not more 
		
		\thanks[footnoteinfo]{This paper was not presented at any IFAC 
			meeting. Corresponding author C.~Sonveaux. Tel. +32-81-724945. }
		
		\author[UN]{Candy Sonveaux}\ead{candy.sonveaux@unamur.be},    % Add the 
		\author[UN]{Joseph J. Winkin}\ead{joseph.winkin@unamur.be},               % e-mail address 
	
		\address[UN]{ University of Namur, Department of Mathematics and naXys, Rue de Bruxelles 61, B-5000 Namur}  % Please supply                                                           % full addresses
		% here.

		                             % keyword list or with the 
		% help of the Automatica 
		% keyword wizard

		\begin{abstract}                          % Abstract of not more than 200 words.
			An age-dependent SIR model is considered with the aim to develop a state-feedback vaccination law in order to eradicate a disease. A dynamical analysis of the system is performed using the principle of linearized stability and shows that, if the basic reproduction number is larger than $1$, the disease free equilibrium is unstable. This result justifies the developement of a vaccination law. Two approaches are used. The first one is based on a dicretization of the partial integro-differential equations (PIDE) model according to the age. In this case a linearizing feedback law is found using Isidori’s theory. Conditions guaranteeing stability and positivity are established.  The second approach yields a linearizing feedback law developed for the PIDE model. This law is deduced from the one obtained for the ODE case.  Using semigroup theory, stability conditions are also obtained.  Finally, numerical simulations are presented to reinforce the theoretical arguments. 
		\end{abstract}
		
	\end{frontmatter}
	
	\section{Introduction}
	Now, even more than before, we know that infectious diseases may lead to huge damage once out of control. The successful eradication of those diseases implies in particular the ability to understand their transmission dynamics. For this purpose, an adapted version of the well-known SIR model of Kermack and McKendrick \cite{kermack_1991} is considered here. Indeed, several adaptations of this model were performed along the time, in particular with models taking into account the age of the individuals (see e.g. \cite{dietz_1985}, \cite{anderson_1983} and \cite{Inaba}). The population is assumed to be divided into three distinct
	classes : the group S of uninfected individuals susceptible
	to catch the disease; the group I of infected individuals
	who can transmit the disease and the group R of recovered
	individuals who are permanently immune to the disease. In the following, we will use the terminology S-, I- and R-individuals, to refer to susceptible infected and recovered individuals, respectively.
	This is the main assumption in a SIR model, which is a
	simple but validated and widespread model.
	Here the importance of the individuals age in the model is taken into account. It is motivated by the fact that several factors in diseases propagation
	depend on the age of the individuals, vaccination being one
	of them. \\In this framework, the dynamics of the disease propagation is described by a set of partial integro-differential equations, as mentioned for instance in \cite{Bastin} and references therein. 
	The dynamical analysis of such systems and more complex ones is well developed in the literature (see e.g. \cite{Inaba1990}, %\cite{anita_2000},
	\cite{Inaba2006}, \cite{Inaba} and \cite{yang_2018}). In those articles, the conclusion about the stability of equilibria is performed by using the principle of linearized stability. However, as far as we can judge, no proof that this principle can be applied is provided. In this paper, a proof of the principle of linearized stability is developed using recent theoretical arguments. Moreover, although the question of control of age-dependent diseases was studied by several authors (see e.g. \cite{Cai}, \cite{demasse_2015} and \cite{tahir_2021}), it is, up to our knowledge, often performed using optimal control methods. Therefore, in those papers, an additional class of individuals is considered: the class of vaccinated individuals. Some authors, as in \cite{liu_2011}, use a pulse vaccination strategy instead of a continuous vaccination law. The particularity of this work is the use of a nonlinear stabilizing linearizing state-feedback control law on a model described by partial integro-differential equations (PIDE). The design of this feedback law is based on the one for an approximate model with ordinary differential equations (ODE). In both cases the global linearizing stability analysis of the feedbacks is performed. Observe that for the PIDE model, this analysis is performed on an infinite dimensional state-space model involving bounded operators.\\
	The paper is organized as follows. In Section \ref{Model}, four versions of an age-structured SIR epidemic model are presented. The three first ones are equivalent modulo a change of variables and consist of a system of nonlinear partial integro-differential equations. The fourth one is an approximation obtained via a discretization of the first system: The disease dynamics of this model consists of a system of nonlinear ordinary differential equations (ODE). The dynamical analysis of the system, detailed in Section \ref{dynamical}, is developed for the PIDE model. Results on well-posedness and stability are obtained and illustrated numerically. Based on those results, the design of a positively stabilizing state-feedback law is performed for the ODE model in Section \ref{positive}. Results on the global stability of the control law and positivity of the solutions are obtained. The aim of this section is to guide the design of a positively stabilizing state-feedback law for the PIDE model, which is described in Section \ref{positive_PDE}. The stabilizing property of this feedback is proven.  In both sections, numerical simulations corroborate the analytical results.  
	\section{Model formulation}\label{Model}
	In this section, an age-dependent SIR epidemic model, described in \cite{Bastin}, is considered. A change of variables, inspired by \cite{Inaba}, is used in order to restrain the model to two partial differential equations instead of three. We also perform a discretization by age of the first model which leads to a set of ordinary differential equations. Moreover, in the framework of epidemic models, several control policies can be studied such as vaccination, quarantine and isolation, treatment, sterilization, slaughter... In this work, we consider vaccination as input to the model.
	\subsection{SIR Model}
	As in the classical SIR model, the population is divided into three distinct classes: the groups $S$ of S-individuals, $I$ of I-individuals and $R$ of R-individuals. \\
	The evolution of these groups, in terms of densities, is described by a system of nonlinear partial integro-differential equations (PIDE model) 
	\begin{align} \label{PDE_SIR}
		\left({{\partial }_t} + \alpha{{\partial }_a}\right) S\left(t,a\right) &= -\left(\Theta\left(t,a\right)+\mu\left(a\right)\right)S\left(t,a\right)\nonumber\\& -\beta\left(a\right) S\left(t,a\right)\displaystyle\int_0^L I\left(t,b\right)db,\nonumber\\ 
		\left({{\partial }_t} + \alpha{{\partial }_a}\right) I\left(t,a\right) &= -\left(\mu\left(a\right)+\gamma\left(a\right)\right)I\left(t,a\right)\\&+\beta\left(a\right) S\left(t,a\right) \displaystyle\int_0^L I\left(t,b\right)db,\nonumber\\
		\left({{\partial }_t} + \alpha{{\partial }_a}\right) R\left(t,a\right) &=\Theta\left(t,a\right)S\left(t,a\right)+\gamma\left(a\right)I\left(t,a\right)\nonumber\\&-\mu\left(a\right)R\left(t,a\right)\nonumber
	\end{align}
	under non-negative initial conditions $S\left(0,a\right)=S_0\left(a\right)$, $I\left(0,a\right)=I_0\left(a\right)$, $R\left(0,a\right)=R_0\left(a\right)$ and boundary conditions $S\left(t,0\right)=B, I\left(t,0\right)=0, R\left(t,0\right)=0$, where $B$ denotes the birth rate.\\
	The interpretations and units of the variables and parameters involved in Model \eqref{PDE_SIR} and in the following ones are described in the table of Appendix \ref{Appendix1}.
	The coefficient $\alpha$ is introduced to balance the possible change of units between time and age. For instance, it is set to $1/365$ when time is in day and age in year, which is the case here.
	The variables $S\left(t,a\right), I\left(t,a\right)$ and $R\left(t,a\right)$ denote the age density of individuals of each group at time $t$. Therefore, the number of S-individuals between two given ages $a$ and $b$ between $0$ and $L$ (the maximum life duration) is given by $\displaystyle\int_{a}^{b} S\left(t,s\right)ds.$
	The population is assumed to be closed, meaning that there is no immigration or emigration. Therefore, a modification in the total size of the population is only caused by birth and mortality. Their respective rates are given by $B$, which is assumed to be constant, and $\mu\left(a\right)$. Moreover, the mode of transmission of the disease is assumed to be by contact between S-individuals and I-individuals and this disease transmission rate is given by $\beta\left(a\right) \displaystyle\int_0^L I\left(t,b\right)db$ where $\beta\left(a\right)$ is the transmission coefficient between S-individuals of age $a$ and all I-individuals. In addition, all the I-individuals become recovered when they are cured. The recovering rate is denoted by $\gamma\left(a\right)$. It is assumed that $\beta\left(\cdot\right)$ and $\gamma\left(\cdot\right)$ are in $L^{\infty}\left(\left[0,L\right]\right)$. Finally, the term $\Theta\left(t,a\right)$ is the input variable representing the rate of S-individuals being vaccinated at time $t$ and age $a$. Those individuals leave the class of S-individuals and become recovered. The vaccination is assumed to work perfectly, meaning that once an individual is vaccinated, he/she gets recovered and never catches the disease afterwards.
	\subsection{Normalized SIR model}
	The age density of the total population is given by $P\left(t,a\right) =S\left(t,a\right)+I\left(t,a\right)+R\left(t,a\right)$. Therefore, the dynamics of the total population is given by
	$$
	\partial_t P\left(t,a\right)+ \alpha\partial_a P\left(t,a\right)=-\mu\left(a\right)P\left(t,a\right),$$
	with initial condition $P\left(0,a\right)=P_0\left(a\right)$ and boundary condition $P\left(t,0\right)=B$.
	The solution of this system can be determined by using the method of characteristics, as mentioned in \cite{Hethcote_2008},
	\begin{equation*}
		P\left(t,a\right)=\left\{\begin{array}{l}
			B\exp\left(-\dfrac{1}{\alpha}\displaystyle\int_0^a\mu\left(\eta\right)d\eta\right) \text{for }t\geq \dfrac{a}{\alpha},\\[0.4cm]
			P_0\left(a-\alpha t\right)\exp\left(-\displaystyle\int_{a-\alpha t}^a \dfrac{1}{\alpha}\mu\left(\eta\right)d\eta\right) \text{otherwise}.
		\end{array}\right.
	\end{equation*}
	\indent In order to get a dimensionless model of System \eqref{PDE_SIR}, new variables, $s$, $i$ and $r$ defined by 
	$S\left(t,a\right)=P\left(t,a\right)s\left(t,a\right)$; $I\left(t,a\right)=P\left(t,a\right)i\left(t,a\right)$; $R\left(t,a\right)=P\left(t,a\right)r\left(t,a\right)$
	are introduced. Therefore, Model \eqref{PDE_SIR} can be rewritten as a normalized nonlinear system of partial integro-differential equations, denoted by NPIDE,
	\begin{align}\label{PDE_SIR_normalised}	
		\left(\partial_t + \alpha\partial_a\right) s\left(t,a\right)&= -\Theta\left(t,a\right)s\left(t,a\right)\nonumber\\&-\beta\left(a\right)s\left(t,a\right)\displaystyle\int_0^L i(t,b)P\left(t,b\right) db, \nonumber\\
		\left(\partial_t + \alpha\partial_a\right) i\left(t,a\right) &= -\gamma\left(a\right)i\left(t,a\right)\\ &+\beta\left(a\right)s\left(t,a\right)\displaystyle\int_0^L i(t,b)P\left(t,b\right) db, \nonumber\\
		\left(\partial_t +\alpha\partial_a\right) r\left(t,a\right) &=\Theta\left(t,a\right)s\left(t,a\right)+  \gamma\left(a\right)i\left(t,a\right)\nonumber
	\end{align}
	under initial conditions $s\left(0,a\right)=s_0\left(a\right),$ $i\left(0,a\right)=i_0\left(a\right),$ $r\left(0,a\right)=r_0\left(a\right)$
	and boundary conditions $s\left(t,0\right)=1,$ $i\left(t,0\right)=0,$ $r\left(t,0\right)=0$.
	Using those variables, $s\left(t,a\right)+i\left(t,a\right)+r\left(t,a\right)=1$. Therefore, only two equations are needed in order to characterize the dynamics of the disease propagation.\\ In some cases, it is easier to work with a system with homogeneous boundary conditions. Here the new variable $\hat{s}\left(t,a\right)=s\left(t,a\right)-1$ yields an equivalent model with homogeneous boundary conditions, which will be denoted as HNPIDE, 
	\begin{align}\label{PDE_SIR_normalised_null_BC}	
		\left(\partial_t + \alpha\partial_a\right) \hat{s}\left(t,a\right)&= -\Theta\left(t,a\right)\left(1+\hat{s}\left(t,a\right)\right)\nonumber\\&+\beta\left(a\right)\left(1+\hat{s}\left(t,a\right)\right)\displaystyle\int_0^L i(t,b)P\left(t,b\right) db, \nonumber\\
		\left(\partial_t + \alpha\partial_a\right) i\left(t,a\right) &= -\gamma\left(a\right)i\left(t,a\right)+\\ &\beta\left(a\right)\left(1+\hat{s}\left(t,a\right)\right)\displaystyle\int_0^L i(t,b)P\left(t,b\right) db, \nonumber 
	\end{align}
	under initial conditions $\hat{s}\left(0,a\right)=\hat{s}_0\left(a\right)=s_0\left(a\right)-1$ $i\left(0,a\right)=i_0\left(a\right)$
	and boundary conditions $\hat{s}\left(t,0\right)=0,$ $i\left(t,0\right)=0.$ \\
	Observe that in the following, we consider the same units in age and time, therefore $\alpha$ equals $1$.
	\subsection{Age Discretized Normalized SIR Model}\label{Section_NODE}
	The previous infinite dimensional models can be approximated by using discretization by age. This yields a nonlinear finite dimensional model for which some known theories can be applied. Inspired by Tudor's article \cite{Tudor}, Model \eqref{PDE_SIR} is discretized in $n$ classes of age, $\left[0,a_1\right), \left[a_1,a_2\right),...,\left[a_{n-1},L\right)$. The proportion of S-individuals in the $k-$th class of age represents the fraction of individuals in class age $k$ that is susceptible at time $t$, which gives $$ s_k\left(t\right)= \dfrac{\displaystyle\int_{a_{k-1}}^{a_k}S\left(t,a\right)da}{N_k} $$ where, assuming that the population has reached a time-invariant age distribution ($P\left(t,a\right)=P\left(a\right)$), $$N_k=\displaystyle\int_0^L P\left(a\right)da$$ corresponds to the total number of individuals in the $k$th class of age in the population. Similar relations hold for the proportion of  I- and R-individuals in the $k-$th class of age at time $t$. Moreover, it is assumed that the continuous functions of age from Model \eqref{PDE_SIR} are constants for a fixed class of age.  In other words, it is assumed that $\mu\left(a\right)=\mu_k, \gamma\left(a\right)=\gamma_k, \beta\left(a\right)=\beta_k$ for $a\in\left[a_{k-1},a_k\right)$ for all $k=1,...,n$.  Note that these constants are taken, in numerical simulations, as the mean values of the considered functions on this interval. However, other choices could be made. Moreover, the input is also assumed to be independent of $a\in\left[a_{k-1},a_k\right)$ and is given by $\theta_k\left(t\right)=\Theta\left(t,a_k\right)$ for all $k=1,...,n$.  In addition, the number of S-individuals that are moving from the $k-$th class of age to the $(k+1)-$th at time $t$, $S\left(t,a_k\right)$, is assumed to be proportional to the size of the $k-$th class of age, i.e there exists a transfer rate $\rho_k$ such that $S\left(t,a_k\right)=\rho_kN_ks_k\left(t\right)$ for $=k...n$. Remark that, since $L$ is the maximum age, $\rho_n$ equals $0$. The transfer rate $\rho_k$ is also used for $I\left(t,a_k\right)$ and $R\left(t,a_k\right)$.\\ As mentioned in \cite{Tudor}, integrating the equations of Model \eqref{PDE_SIR} with respect to the age variable, from $a_{k-1}$ to $a_k$, for $k=1,...,n$ and using previous assumptions and initial conditions of Model \eqref{PDE_SIR}, lead to the following system of nonlinear ODEs, \vspace{-0.6cm}
	\small{
	\begin{align*}
		N_k\dfrac{ds_k\left(t\right)}{dt}&=\rho_{k-1}N_{k-1}s_{k-1}\left(t\right)-\left(\rho_k+\mu_k+\theta_k\left(t\right)\right)N_ks_k\left(t\right)\\&-\beta_kN_ks_k\left(t\right)\displaystyle\sum_{j=1}^n N_ji_j\left(t\right)+B\left(t\right)\delta_{1k}\\
		N_k\dfrac{di_k\left(t\right)}{dt}&=\rho_{k-1}N_{k-1}\left(t\right)i_{k-1}\left(t\right)-\left(\rho_k+\gamma_k+\mu_k\right)N_ki_k\left(t\right)\\&+\beta_kN_ks_k\left(t\right)\displaystyle\sum_{j=1}^n N_ji_j\left(t\right)\\
		N_k\dfrac{dr_k\left(t\right)}{dt}&=\rho_{k-1}N_{k-1}r_{k-1}\left(t\right)+\gamma_kN_ki_k\left(t\right)\\&+\theta_k\left(t\right)N_ks_k\left(t\right)-\left(\rho_k+\mu_k\right)N_kr_k\left(t\right)
	\end{align*}}for $k=1,...,n$ where $\rho_0$ is chosen to be equal to $0$ and $\delta_{ij}$ denotes the Kronecker symbol.\\
	Summing those equations gives the following relations for the $N_i$'s,\vspace{-0.1cm}
	\begin{align*}
		\dfrac{B\left(t\right)}{N_1}&=\rho_1+\mu_1,\\
		\rho_{k-1}\dfrac{N_{k-1}}{N_k}&=\rho_k+\mu_k, \text{ for }k=2,...,n.
	\end{align*}
	This leads to the identity $s_k+i_k+r_k=1$, $k=1,...,n$. Therefore, only $2n$ equations are needed.
	Moreover, using this last assumption and the previous relations, and dividing the set of ODE's equations by $N_k$ for $k=1,...,n$ gives a set of $2n$ ordinary differential equations: 
	\begin{align}	
		\label{ODE}	
		\dfrac{ds_k\left(t\right)}{dt}&=T_ks_{k-1}\left(t\right)\nonumber\\&-\left(T_k +\theta_k\left(t\right)+\beta_k\displaystyle\sum_{j=1}^n N_ji_j\left(t\right)\right)s_k\left(t\right),\\
		\dfrac{di_k\left(t\right)}{dt}&=T_ki_{k-1}\left(t\right)-\left(T_k+\gamma_k\right)i_k\left(t\right)\nonumber\\&+\beta_ks_k\left(t\right)\displaystyle\sum_{j=1}^n N_ji_j\left(t\right)\nonumber
	\end{align}
	for $k=1,...,n,$ where $T_k=\rho_k+\mu_k$ and by setting $s_0\left(t\right)=1$ and $i_0\left(t\right)=0$ .
	In the following, this model will be called the NODE Model since it is an age discretized normalized model (involving proportions as variables). 
	\section{Dynamical Analysis of HNPIDE Model}\label{dynamical}
	\subsection{Well-posedness and stability of equilibria}
	Using similar arguments as in \cite{Inaba1990} and in \cite[Chap.6]{Inaba}, Model \eqref{PDE_SIR_normalised_null_BC} is well-posed, assuming that the rate of vaccinated S-individuals is given by a Lipschitz continuous state feedback law $\Theta\left(t,a\right)=F\left(\hat{s}\left(t,a\right),i\left(t,a\right)\right)$, since the existence and uniqueness of a non-negative solution $\left(\hat{s},i\right)$ which is smaller than $1$ can be proven using semigroup theory and the method of characteristics. \\
	Moreover, regarding the stability analysis, notice that $\displaystyle\lim_{t \to \infty} P\left(t,a\right)=B\exp\left(-\displaystyle\int_0^L\mu\left(\eta\right)d\eta\right)=:c\left(a\right).$ Therefore, Model \eqref{PDE_SIR_normalised} is asymptotically autonomous, then the stability analysis can be performed on the limiting autonomous normalized system, which is Model \eqref{PDE_SIR_normalised} where $P\left(t,b\right)$ is replaced with $c\left(b\right)$. Secondly, notice that, at the equilibrium, the input $\Theta\left(t,a\right)$ does not depend on time and is denoted by $\Theta^{\star}\left(a\right)$.\\ Different conclusions are obtained according to the value of the basic reproduction number of infection given by $$R(0)=\displaystyle\int_0^L c\left(b\right)\Gamma\left(b\right)\displaystyle\int_0^b \dfrac{\beta\left(\sigma\right)}{\Gamma\left(\sigma\right)}\exp\left(-\displaystyle\int_0^{\sigma}\Theta^{\star}\left(\eta\right)d\eta\right)d\sigma db,$$ where  $\Gamma\left(b\right)=\exp\left(-\displaystyle\int_0^b \gamma\left(\eta\right)d\eta\right)$. If $R\left(0\right)\leq 1,$ there is only one epidemic steady-state, the disease-free equilibrium, $\left(s^{\star}, i^{\star}\right)=\left(\exp\left(-\displaystyle\int_0^{a}\Theta^{\star}\left(\eta\right)d\eta\right),0\right)$. Otherwise, if $R\left(0\right)>1$, there are two endemic steady-states, one corresponding to the disease-free equilibrium and one endemic equilibrium. The stability of equilibra is performed by studying the linearized system around the equilibrium $\left(s^{\star}, i^{\star}\right)$ denoted by $x_e$ in what follows. A proof showing that the principle of linearized stability can be applied here is developed. This principle shows that, under some hypothesis, the stability of the linearized system implies the local stability of the considered equilibrium for the nonlinear system. A proof of this principle, used without proof in \cite{Inaba2006}, is performed and detailed using recent theoretical arguments. The proof is based on a particular case (with space $X$=$Y$) of a result developed in \cite[Theorem 9]{hastir_2020c}, that extends results of \cite{Morris_2018}, on Banach spaces.
	\begin{lemma}\cite[Theorem 9]{hastir_2020c}
		\label{Hastir}
		Consider a semilinear system of the form \begin{equation}\label{NL}
			\left\{\begin{array}{l}
				\dot{x}=\mathcal{A}x+\mathcal{N}\left(x\right)\\
				x\left(0\right)=x_0
			\end{array}\right.
		\end{equation}
		where $\mathcal{A}$ is a linear operator on its domain $\mathcal{D}\left(\mathcal{A}\right)$, which is a linear subspace of a Banach spaxe $X$, and $\mathcal{N}$ is a nonlinear operator such that $\mathcal{N} : \mathcal{D}\left(\mathcal{A}\right)\cap\mathcal{D}\left(\mathcal{N}\right)\subset X \to X$.\\
		Assume that \eqref{NL} admits an equilibrium profile $x_e$, i.e there exists $x_e \in \mathcal{D}\left(A\right)\cap \mathcal{D}\left(N\right)$ such that $$ \mathcal{A}x_e+\mathcal{N}\left(x_e\right)=0.$$\\
		Assume that the following conditions hold: $\mathcal{A}$ is quasidissipative, i.e. there exists $l_{\mathcal{A}}>0$ such that the operator $\mathcal{A}-l_{\mathcal{A}}I$ is dissipative on $\mathcal{D}\left(\mathcal{A}\right)\cap\mathcal{D}\left(\mathcal{N}\right)$; the nonlinear operator $\mathcal{N}$ is Lipschitz continuous on $\mathcal{D}\left(\mathcal{A}\right)\cap\mathcal{D}\left(\mathcal{N}\right)$ with respect to the $X$ norm; the operator $\mathcal{A}+\mathcal{N}$ is the infinitesimal generator of a nonlinear $C_0-$semigroup $\left(S\left(t\right)\right)_{t\geq 0}$ on $X$; the Gâteaux derivative $d\mathcal{N}\left(x_e\right)$ of $\mathcal{N}$ at $x_e$ is a bounded linear operator on $X$, the Gâteaux linearized dynamics of \eqref{NL} is given by \begin{equation}\label{L}
			\left\{\begin{array}{l}
				\dot{\bar{x}}=\left(\mathcal{A}+d\mathcal{N}\left(x_e\right)\right)\bar{x}\\
				\bar{x}\left(0\right)=x_0-x_e=\bar{x}_0
			\end{array}\right.
		\end{equation} and the nonlinear semigroup $\left(S\left(t\right)\right)_{t\geq 0}$ is Fréchet differentiable with Fréchet derivative $\left(T_{x_e}\left(t\right)\right)_{t\geq 0}$ corresponding to the linear semigroup generated by the Gâteaux derivative of $\mathcal{A}+\mathcal{N}$ at $x_e$.\\
		Then, if $x_e$ is a (globally) exponentially stable equilibrium of the linearized system \eqref{L}, then it is a locally exponentially stable equilibrium\footnote{As defined in \cite{Hastir_2020_bis}, $x_e$ is a locally exponentially stable equilibrium if $\exists$ $ \alpha, \beta, \delta >0$ s.t $\forall$ $x_0\in \mathcal{D}\left(\mathcal{A}\right)\cap\mathcal{D}\left(\mathcal{N}\right) :\\ \Vert x_0-x_e \Vert < \delta \Rightarrow \Vert x\left(t\right)-x_e \Vert \leq \alpha e^{-\beta t}\Vert x_0-x_e\Vert.$} of \eqref{NL}. Moreover, if $x_e$ is an unstable equilibrium of \eqref{L}, it is locally unstable for the nonlinear system \eqref{NL}.
	\end{lemma}    
	\begin{thm}\label{linearised_stab}
		If $x_e$ is a (globally) exponentially stable equilibrium of the linearization of the nonlinear HNPIDE Model \eqref{PDE_SIR_normalised_null_BC}, then it is a locally exponentially stable equilibrium of Model \eqref{PDE_SIR_normalised_null_BC}. Moreover, if $x_e$ is an unstable equilibrium of the linearization of Model \eqref{PDE_SIR_normalised_null_BC}, it is locally unstable for Model \eqref{PDE_SIR_normalised_null_BC}.
	\end{thm}
	\begin{proof}
	Model \eqref{PDE_SIR_normalised_null_BC} with the input at equilibrium can be rewritten as the abstract differential equation \begin{equation*}
		\left\{\begin{array}{l}
			\dot{\hat{x}}=\mathcal{A}\hat{x}+\mathcal{N}\left(\hat{x}\right)\\
			\hat{x}\left(0\right)=\hat{x}_0
		\end{array}\right.
	\end{equation*} where $\hat{x}=\left(
	\hat{s}, i
	\right)^T$, $\mathcal{A}=-\dfrac{d\cdot}{da}I_2$
	with $I_2$ the identity matrix of dimension $2$, \\$\mathcal{D}\left(\mathcal{A}\right)=\left\{\hat{x}\in X : \hat{s}, i \in AC\left[0,L\right] \text{ and } \hat{s}\left(0\right)=i\left(0\right)=0\right\}$ where $X=L^1\left(0,L\right)\times L^1\left(0,L\right)$. Note that $X$ is a Banach space and its norm is defined for all $x=\left(x_1,x_2\right)^T \in X$ by $\Vert x \Vert_X:=\Vert x \Vert=\Vert x_1 \Vert_1 + \Vert x_2\Vert_1$ where $\Vert x \Vert_1$ is the usual norm on $L^1\left(0,L\right)$. Moreover $\mathcal{N} : \mathcal{D}\left(\mathcal{N}\right) \rightarrow X $ is defined for all $\hat{x}\in X$ by $$\mathcal{N}\left(\hat{x}\right)=\begin{pmatrix}
		-\left(\Theta^{\star}\left(\cdot\right)+\beta\left(\cdot\right)\displaystyle\int_0^L i(b)c\left(b\right) db\right)\left(1+\hat{s}\left(\cdot\right)\right)\\
		-\gamma\left(\cdot\right)i\left(\cdot\right)+\beta\left(\cdot\right)\left(1+\hat{s}\left(\cdot\right)\right)\displaystyle\int_0^L i(b)c\left(b\right) db
	\end{pmatrix}$$\\
where \\ $\mathcal{D}\left(\mathcal{N}\right)=\left\{\hat{x}\in X : -1\leq \hat{s} \leq 0, 0\leq i \leq 1 \text{ a.e. on }[0,L]\right\}.$
	These operators satisfy the hypothesis of Lemma \ref{Hastir}. Indeed, let $\lambda>0$ and $\hat{x}=\left(\hat{s},i\right)^T \in \mathcal{D}\left(\mathcal{A}\right)\cap\mathcal{D}\left(\mathcal{N}\right)$ be arbitrarily fixed, knowing that $-1\leq \hat{s}<0$ and $0\leq i \leq 1$, we get the following inequalities for any $l_{\mathcal{A}}>0$: \begin{align*}
		\Vert \left(\lambda I - \mathcal{A} + l_{\mathcal{A}}I\right)\hat{x} \Vert_X &= \displaystyle\int_0^L \vert \left(\lambda+ l_{\mathcal{A}}\right)\hat{s}\left(a\right)+\frac{d\hat{s}\left(a\right)}{da} \vert da \\& +\displaystyle\int_0^L \vert \left(\lambda+ l_{\mathcal{A}}\right)i\left(a\right)+\frac{di\left(a\right)}{da} \vert da \\
		&\geq \vert \displaystyle\int_0^L \left(\lambda+ l_{\mathcal{A}}\right)\hat{s}\left(a\right) da \vert \\& + \vert \displaystyle\int_0^L \left(\lambda+ l_{\mathcal{A}}\right)i\left(a\right) da \vert \\
		& \geq \lambda \Vert \hat{x} \Vert_X.
	\end{align*}
	Moreover, using the fact that $F$ is assumed to be Lipschitz continuous, it can be shown that $\mathcal{N}$ is Lipschitz continuous on $\mathcal{D}\left(\mathcal{A}\right)\cap\mathcal{D}\left(\mathcal{N}\right)$. Using Theorem 1.2 from \cite[Chap. 6, Sect. 1]{Pazy}, we conclude the existence of a mild solution $\hat{x}\left(t\right)=S\left(t\right)\hat{x}_0$ for all $t\geq 0$.
	Furthermore, the Gâteaux derivative of $\mathcal{N}$ at $\hat{x}_e=\left(\hat{s}^{\star}\left(a\right),i^{\star}\left(a\right)\right)^T=\left(s^{\star}\left(a\right)-1,i^{\star}\left(a\right)\right)^T$ is given by \begin{align*}
		d\mathcal{N}&\left(\hat{x}_e\right)z=\displaystyle\lim\limits_{\epsilon\to 0} \dfrac{\mathcal{N}\left(\hat{x}_e+\epsilon z\right)-\mathcal{N}\left(\hat{x}_e\right)}{\epsilon}\\&=\begin{pmatrix}
			-\beta\left(\cdot\right)\left(1+\hat{s}^{\star}\left(\cdot\right)\right)\lambda\left(y\right)-x\left(\Theta^{\star}\left(\cdot\right)+\beta\left(\cdot\right)\lambda\left(i^{\star}\right)\right)\\
			-\gamma\left(\cdot\right)+\beta\left(\cdot\right)\left(1+\hat{s}^{\star}\left(\cdot\right)\right)\lambda\left(y\right)+x\beta\left(\cdot\right)\lambda\left(i^{\star}\right)
		\end{pmatrix} 
	\end{align*} for all $z=\left(x,y\right)^T\in X$, where $\lambda\left(z\right)=\displaystyle\int_0^Lz\left(b\right)P\left(b\right) db$.
	Using the fact that $\gamma\left(\cdot\right)$ and $\beta\left(\cdot\right)$ are bounded and $1+\hat{s}^{\star}\left(a\right)\leq 1$ we can show that $d\mathcal{N}\left(\hat{x}_e\right)$ is bounded. Moreover it is a linear operator. To prove the last assumption about the Fréchet differentiability of the nonlinear semigroup, it suffices to prove that the nonlinear operator $\mathcal{N}$ is Fréchet differentiable at $\hat{x}_e$ and that $\left(S\left(t\right)\right)_{t\geq 0}$ depends continuously on the initial conditions. $\mathcal{N}$ is Fréchet-differentiable at $\hat{x}_e$ if there exists a bounded linear operator $DN\left(\hat{x}_e\right):X\to X$ such that, for all $h=\left(h_1,h_2\right)^T\in X$, $\displaystyle\lim_{\Vert h\Vert \to 0}\dfrac{\Vert \mathcal{N}\left(\hat{x}_e+h\right)-\mathcal{N}\left(\hat{x}_e\right)-DN\left(\hat{x}_e\right)h\Vert}{\Vert h \Vert}=0$. It can be shown that $d\mathcal{N}\left(\hat{x}_e\right)$ is convenient. Indeed, $\Vert \mathcal{N}\left(\hat{x}_e+h\right)-\mathcal{N}\left(\hat{x}_e\right)-d\mathcal{N}\left(\hat{x}_e\right)h\Vert$\\ $ = 2\displaystyle\int_0^L\vert \beta\left(a\right)h_1\left(a\right)\lambda\left(h_2\right) \vert da.$ Dividing this quantity by $\Vert h \Vert$, we can show that it is smaller than $K\Vert h_1 \Vert_1$ which tends to $0$. Therefore, $\mathcal{N}$ is Fréchet differentiable. Moreover, considering the change of variables $\tilde{x}=\hat{x}-\hat{x}_e$, we obtain the following abstract differential equation:  \begin{equation*}
		\left\{\begin{array}{l}
			\dot{\tilde{x}}=\mathcal{A}\hat{x}+\tilde{\mathcal{N}}\left(\tilde{x}\right)\\
			\tilde{x}\left(0\right)=\tilde{x}_0
		\end{array}\right.
	\end{equation*} where $\tilde{\mathcal{N}}:X \to X$ is given by $\tilde{\mathcal{N}}\left(\tilde{x}\right)=\mathcal{N}\left(\tilde{x}+\hat{x}_e\right)-\mathcal{N}\left(\hat{x}_e\right)$ and $A$ is the infinitesimal generator of a semigroup of contraction $\left(T\left(t\right)\right)_{t\geq 0}.$ Therefore, with $\hat{x}(t)=S(t)\hat{x}_0$, \begin{align*}
		\Vert \hat{x}(t)\Vert &\leq \Vert T\left(t\right)\tilde{x}_0\Vert + \displaystyle\int_0^t \Vert T\left(t-s\right) \Vert \Vert \tilde{\mathcal{N}}\left(s,\tilde{x}\left(s\right)\right) \Vert ds\\
		&\leq \Vert \tilde{x}_0\Vert + \displaystyle\int_0^t \Vert \mathcal{N}\left(\tilde{x}+\hat{x}_e\right)-\mathcal{N}\left(\hat{x}_e\right) \Vert ds\\
		&\leq \Vert \tilde{x}_0\Vert + K \displaystyle\int_0^t \Vert \tilde{x}(s) \Vert ds.\\
		\text{Hence,}\\
		\Vert \hat{x}(t)\Vert&\leq \Vert \tilde{x}_0\Vert e^{Kt},
	\end{align*}
	by Grönwall's inequality.
	Lemma \ref{Hastir} concludes the proof. 
		\end{proof}\\ \\
	Using Theorem \ref{linearised_stab}, semigroup theory and property of operators (analytic operators, compact operators, non-supporting operators, ...), the stability of the equilibria is obtained (see \cite{Inaba1990} for details). 
	\begin{coro}\label{coro1}
		The disease free equilibrium is locally exponentially stable when $R\left(0\right)\leq 1$ but is locally exponentially unstable if $R\left(0\right)> 1$ while the endemic equilibrium is locally exponentially stable if $R\left(0\right)> 1$.  
	\end{coro}
	
	\subsection{Numerical Simulations}\label{NS}
	Results of the previous subsection are confirmed using numerical simulations where no control is considered. Most of Parameters are taken from \cite{Okuwa} where the PIDE Model is used. First, note that $L$ is fixed to $1$ in order to normalize the age interval as $\left[0,1\right)$. The age-specific death-rate is given by $\mu(a)=\left(10\left(1-a\right)^2\right)^{-1}$ with $a\in\left[0,1\right[$. Therefore, $l(a)=\exp\left(-a/\left(10\left(1-a\right)\right)\right), a\in\left[0,1\right).$ If $B$ is chosen equals to $1/\int_0^1 l(a)da$, $B$ equals $1.2527$. Then, the total population is normalized ($\int_0^1 N(a)da=1$). In addition, the age-dependent recovery rate is defined by $$\gamma\left(a\right)=100.$$
To be consistent with the following analytical developments, the transmission coefficient used is not the one in \cite{Okuwa}. Indeed, in Section \ref{positive_PDE},
	$\dfrac{\gamma(a)+\mu(a)}{\beta(a)}$ needs to be in $C^1\left[0,L\right]$. Therefore $\beta(a)$ has to be differentiable for all $a \in \left[0,L\right]$ which is not the case in $0$ for the choice in \cite{Okuwa}. Therefore 	the transmission coefficient is defined as $$\beta(a)=\beta_0\left(\sin(a)e^{-2a}+\dfrac{1}{100}\right). $$
	with $\beta_0=600$ or $800$.
	Moreover, the parameters that are used in the numerical simulations are listed in Table \ref{Okuwa_param}. Observe that the age is normalized. Hence there are no units for the age and no units are mentioned for the time.
	\begin{table}[h]
		\centering
		\begin{tabular}{lll}
			\hline
			Parameter & Symbol & Value \\ \hline
			Maximum age & $L$ & $1$  \\
		%	Recovery rate & $\gamma(a)$ & $100$  \\
			Time frame & $T$ & $20$ \\
			Time stepsize & $\delta t$  & $0.001$  \\
			Age stepsize & $\delta a$ & $0.01$\\ \hline
		\end{tabular}
		\caption{Model parameters and values}
		\label{Okuwa_param}
	\end{table}
	Finally, the initial conditions are sligthly modified from \cite{Okuwa} in order to maintain consistency between initial conditions and boundary conditions (i.e when $a$ and $t$ are equal to $0$.) Therefore, we set \begin{align*}
		s_0(a)&=1-i_0(a),\\ i_0(a)&=\left\{\begin{array}{ll}
			\hat{i}_0\left(a\right)-\hat{i}_0\left(0\right) &\text{ if } i_0\left(a\right)\geq 0\\
			0 &	\text{ else,}
		\end{array}\right.\\
		r_0(a)&=0
	\end{align*} where   $$
	\hat{i}_0\left(a\right)=\dfrac{1}{2}e^{-100\left(a-\dfrac{1}{2}\right)^2}\times 10^{-3}.$$The numerical method used in simulation is a forward time - backward space finite difference scheme. The stability of this scheme is ensured by the necessary and sufficient conditions of Courant-Friedrichs-lewy which requires in this case that $\left|\dfrac{\delta t}{\delta a}\right| \leq 1,$ as mentioned in \cite{alexanderian_2011}.  \\
	First, note that similar results as the ones shown in Figures \ref{fig2} and \ref{fig3}, were obtained using Model \eqref{PDE_SIR}. Thus both systems can be used interchangeably. Second, in Figure \ref{fig2}, we can observe that the dynamics of (the proportion, $Nb_I(t,a)$, of) I-individuals\footnote{This quantity is obtained by integrating the density $i(t,b)$ on the intervals $\left[a_{k-1},a_k\right)$ for k=1,...,n. Moreover, all other figures also depict proportions of individuals.} tends to 0 as time increases. This is consistent with the fact that there is only one stable equilibrium when $R\left(0\right)\leq 1$, which is the disease-free equilibrium. Contrariwise, in Figure \ref{fig3} the dynamics of I-individuals tends to an endemic equilibrium where there are still I-individuals in the population when time increases. 
	%\begin{figure}
	%\begin{center}
	%%\includegraphics[height=5cm]{figures/infected_Non_Normalised_b0_380.eps}    % The printed column  
	%%\caption{Dynamics from Model \ref{PDE_SIR} for $R\left(0\right)=0.9566$ }  % width is 8.4 cm.
	%\label{fig1}                                 % Size the figures 
	%\end{center}                                 % accordingly.
	%\end{figure}
	\begin{figure}
		\begin{center}
			\includegraphics[height=5cm]{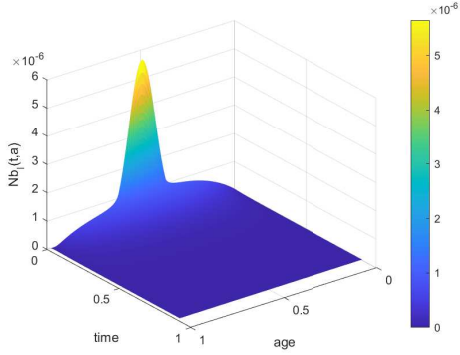}    % The printed column  
			\caption{Dynamics of I-individuals from the NPIDE Model without control for $R\left(0\right)=0.8894$ }  % width is 8.4 cm.
			\label{fig2}                                 % Size the figures 
		\end{center}                                 % accordingly.
	\end{figure}
	\begin{figure}
		\begin{center}
			\includegraphics[height=5cm]{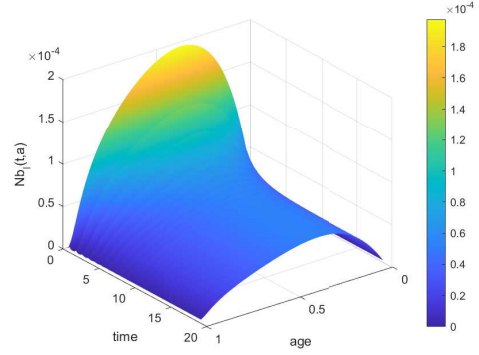}    % The printed column  
			\caption{Dynamics of I-individuals from the NPIDE Model without control for $R\left(0\right)=1.1859$ }  % width is 8.4 cm.
			\label{fig3}                                 % Size the figures 
		\end{center}                                 % accordingly.
	\end{figure}

	\section{Positive closed-loop stabilization of NODE model}\label{positive}
	
	In view of the dynamical analysis of the open loop system, it seems natural to want to stabilize the disease-free equilibrium when $R\left(0\right)>1$ since this equilibrium is unstable for the HNPIDE model (see Corollary \ref{coro1}) and we want to eradicate the disease. \\
	In the following, the aim is to design a feedback control law of vaccination $\Theta\left(t,a\right)$ such that, when it is applied, the corresponding state trajectory converges towards the disease-free equilibrium. \\
	Two vaccination laws are designed. The first one, detailed here, uses Isidori's theory on "Nonlinear Feedback for Multi-Input Multi-Output Systems" developed in \cite[Chap. 5]{Isidori} applying on finite dimensional systems. The second one, explained in next section, is deduced from the first design but acting on the infinite dimensional system. \\
	The current section is inspired on the methodology developed in \cite{Quesada} for SEIR model without age-dependency. We focus on the influence of the age of individuals, given a set of $2n$ ODE, that will be an intuition to solve the PIDE problem.
	\subsection{Model in normal form}
	The aim of this section is to use a coordinate change in order to write the model in normal form as stated in Isidori's theory \cite{Isidori}. The dynamics equations of the NODE Model \eqref{ODE} can be written equivalently in the state space form as a nonlinear control affine system 
	\begin{subequations}
		\label{state_space}
		\begin{align}
			\hspace{-0.3cm}\left\{\begin{array}{lp{0.2cm}l}
				\dot{x}\left(t\right)&=&f\left(x\left(t\right)\right)+g\left(x\left(t\right)\right)u\left(t\right)\\
				y\left(t\right)&=&h\left(x\left(t\right)\right)
			\end{array}\right.
			\intertext{where $x\left(t\right)=\left[i_1\left(t\right),...,i_n\left(t\right), s_1\left(t\right),...,s_n\left(t\right)\right]^T\in \mathbb{R}^{2n},$ for all $ t \geq 0$ is the state space vector, $h\left(x\left(t\right)\right)=\left[i_1\left(t\right),...,i_n\left(t\right)\right]^T\in\mathbb{R}^{n}, \forall t\geq 0$ is the measurable output function, assumed equals to the infectious population and $u\left(t\right)=\left[\theta_1\left(t\right),...,\theta_n\left(t\right)\right]^T\in\mathbb{R}^{n}, \forall t\geq 0$ is the input function. Moreover,}
			g\left(x\left(t\right)\right)=\begin{pmatrix}
				0_{n\times n}\\
				- diag(s_k)_{k=1,...,n}
			\end{pmatrix}
			\intertext{and} f\left(x\left(t\right)\right):=\left(
			f_1\left(x\left(t\right)\right) 
			\cdots
			f_{2n}\left(x\left(t\right)\right)
			\right)^T
		\end{align}
	\end{subequations}
	where
	\begin{align*}
		f_k\left(x\left(t\right)\right)&=
		T_ki_{k-1}\left(t\right)-\left(T_k+\gamma_k\right)i_k\left(t\right)\\ &+\beta_ks_k\left(t\right)\displaystyle\sum_{j=1}^n N_ji_j\left(t\right),\\
		f_{n+k}\left(x\left(t\right)\right)&=
		T_ks_{k-1}\left(t\right)-T_ks_k\left(t\right)\\&-\beta_ks_k\left(t\right)\displaystyle\sum_{j=1}^n N_ji_j\left(t\right)
	\end{align*} for $k=1,...,n$.\\ 
	Since the relative degree of the system equals the dimension of the state space for any $x\in\mathcal{D}=\Bigg\{x \text{ s.t } s_k\left(t\right)\neq 0\text{ for }$ \\$ \left. k=1,...,n \text{ and }\displaystyle\sum_{j=1}^n N_ji_j\left(t\right)\neq 0, t\geq 0 \right\}$, the nonlinear invertible coordinate change that is needed here is given by
	
	\begin{align}
		\label{change_of_variable}
		\bar{i}_k\left(t\right)&=h_k\left(x\left(t\right)\right)=i_k\left(t\right),\nonumber\\
		\bar{s}_k\left(t\right)&=L_fh_k\left(x\left(t\right)\right)\nonumber\\&=f_k\left(x\left(t\right)\right)\\
		&=T_ki_{k-1}\left(t\right)-\left(T_k+\gamma_k\right)i_k\left(t\right)\nonumber\\ &\hspace{0.3cm}+\beta_ks_k\left(t\right)\displaystyle\sum_{j=1}^n N_ji_j\left(t\right)\nonumber
	\end{align}	
	
	for $k=1,...,n$.
	Thanks to this coordinate change, the system is written in its normal form in the neighborhood of any $x\in\mathcal{D}$ by
	\begin{equation}\label{normal_form}
		\begin{aligned}
			\dfrac{d\bar{i}_k\left(t\right)}{dt}&=\bar{s}_k\left(t\right),\\
			\dfrac{d\bar{s}_k\left(t\right)}{dt}&=L_f^2h_k\left(x\left(t\right)\right)+L_{g_k}L_fh_k\left(x\left(t\right)\right)u_k\left(t\right)
		\end{aligned}
	\end{equation}
	for $k=1,...,n$, where $L_a^k b$ is  the $k$ th order Lie derivative of $b$ along the vector field $a$, as defined in \cite{Isidori}.
	\subsection{Feedback design}
	This section aims to design a feedback law that linearizes and stabilizes the system in normal form and implies the eradication of the epidemic from the population.\\ \\
	In order to design the linearizing feedback, the following matrices are defined,
	\begin{align}
		A\left(x\left(t\right)\right)&=diag\left(-\beta_ks_k\left(t\right)\displaystyle\sum_{j=1}^n N_ji_j\left(t\right)\right)_{k=1,...,n}\label{A}\\
		v\left(x\left(t\right)\right)&=\begin{pmatrix}
			v_1\left(x\left(t\right)\right) &\cdots & v_n\left(x\left(t\right)\right)
		\end{pmatrix}^T\label{v}\\
		\intertext{ such that } v_k\left(x\left(t\right)\right)&=-\alpha_2^kf_k\left(x\left(t\right)\right)-\alpha_1^ki_k\left(t\right), k=1,...,n\nonumber
		\intertext{where $\alpha_1^k$ and $\alpha_2^k$ are some free parameters that will be ajust to have stability. Moreover,}
		b\left(x\left(t\right)\right)&=\begin{pmatrix}
			b_1\left(x\left(t\right)\right) & \cdots & b_n\left(x\left(t\right)\right)
		\end{pmatrix}^T,\label{b}
		\intertext{where}  b_k\left(x\left(t\right)\right)&=L_f^2h_k\left(x\left(t\right)\right)\\&=\beta_ks_k\left(t\right)\displaystyle\sum_{j=1}^n N_j f_j\left(x\left(t\right)\right)\nonumber\\&+T_kf_{k-1}\left(x\left(t\right)\right)-\left(T_k+\gamma_k\right)f_k\left(x\left(t\right)\right)\nonumber\\&+\beta_k f_{n+k}\left(x\left(t\right)\right)\displaystyle\sum_{j=1}^n N_ji_j\left(t\right).\nonumber
	\end{align}
	
	\begin{lemma}\label{lemma_fdb}
		The state feedback control law defined by
		\begin{equation}
			u\left(t\right)=A^{-1}\left(x\left(t\right)\right)\left(v\left(x\left(t\right)\right)-b\left(x\left(t\right)\right)\right),
			\label{control}\end{equation} where $A$, $b$ and $v$ are given by \eqref{A}-\eqref{b}, applied on system \eqref{state_space}, induces the linear output closed-loop dynamics given by \begin{align}
			\ddot{y}\left(t\right)+\tilde{A}_2\dot{y}\left(t\right)+\tilde{A}_1y\left(t\right)=0.
			\label{CLdynamics}
		\end{align}
		where $\tilde{A}_1=diag(\alpha_1^i)$ and  $\tilde{A}_2=diag(\alpha_2^i)$ for $i=1,...,n$.
	\end{lemma}
	\begin{proof}
	According to Isidori's theory, the control law defined in \eqref{control} is obtained. It can be rewritten as 
	
	$u\left(t\right)=\begin{pmatrix}
		\dfrac{v_1-L_f^2h_1\left(x\left(t\right)\right)}{L_{g_1}L_fh_1\left(x\left(t\right)\right)} &
		\cdots &
		\dfrac{v_n-L_f^2h_n\left(x\left(t\right)\right)}{L_{g_n}L_fh_n\left(x\left(t\right)\right)}
	\end{pmatrix}^T.$
	
	Therefore, applying this control law to the dynamics in normal form \eqref{normal_form} linearizes the equations and gives  for $k=1,...,n$,
	\begin{align}
		\label{normal_form_linearise}
		\dfrac{d\bar{i}_k\left(t\right)}{dt}&=\bar{s}_k\left(t\right)\nonumber\\
		\dfrac{d\bar{s}_k\left(t\right)}{dt}&=v_k\left(t\right),\\
		&=-\alpha_2^k\bar{s}_k-\alpha_1^k\bar{i}_k\nonumber
	\end{align}
	which can be written as
	\begin{align*}
		\dot{\bar{x}}\left(t\right)&=\begin{pmatrix}
			0_{n \times n} & I_n\\
			-\tilde{A}_1 & -\tilde{A}_2
		\end{pmatrix}\bar{x}\left(t\right),\\
		&:=\bar{A}\bar{x}\left(t\right)
	\end{align*}
	where $\bar{x}\left(t\right)=\left[\bar{i}_1\left(t\right) \cdots \bar{i}_n\left(t\right) \bar{s}_1\left(t\right) \cdots \bar{s}_n\left(t\right) \right]^T$.
	The solution of this ODE is given by\\ $y\left(t\right)=C\bar{x}\left(t\right)=Ce^{\bar{A}t}\bar{x}\left(0\right)=\left(
	\bar{i}_1\left(t\right) \cdots  \bar{i}_n\left(t\right)\right)^T$.\\ Thus, $\dot{y}\left(t\right)=\left(
	\dfrac{d\bar{i}_1\left(t\right)}{dt} \cdots  \dfrac{d\bar{i}_n\left(t\right)}{dt}\right)^T$ and 
	\begin{align*}
		\ddot{y}\left(t\right)&=C\bar{A}^2e^{\bar{A}t}\bar{x}\left(0\right),\\&=\begin{pmatrix}
			-\alpha_1^1 \bar{i}_1\left(t\right)-\alpha_2^1\bar{s}_1\left(t\right) &
			\cdots &
			-\alpha_1^n \bar{i}_n\left(t\right)-\alpha_2^n\bar{s}_n\left(t\right)\\
		\end{pmatrix}^T,\\
		&=\tilde{A}_1y\left(t\right)+\tilde{A}_2\dot{y}\left(t\right).
	\end{align*}
		\end{proof}\\ \\
	Therefore, the feedback law \eqref{control} is linearizing for the model in normal form, which is in adequacy with Isidori's theory. Using this feedback on system \eqref{state_space}, the closed-loop model is given by
	\begin{align}	
		\label{CL}	
		\dfrac{di_k\left(t\right)}{dt}&=T_ki_{k-1}\left(t\right)-\left(T_k+\gamma_k\right)i_k\left(t\right)\nonumber\\&+\beta_ks_k\left(t\right)\displaystyle\sum_{j=1}^n N_ji_j\left(t\right),\nonumber\\
		\dfrac{ds_k\left(t\right)}{dt}&=\dfrac{1}{\beta_k\displaystyle\sum_{j=1}^n N_ji_j\left(t\right)}\left(f_k\left(x\left(t\right)\right)\left(T_k+\gamma_k-\alpha_2^k\right)\right.\\&\left.-\alpha_1^k i_k\left(t\right)-T_kf_{k-1}\left(x\left(t\right)\right)-\right.\nonumber\\&\left.\beta_ks_k\left(t\right)\displaystyle\sum_{j=1}^n N_j f_j\left(x\left(t\right)\right)\right)\nonumber
	\end{align}
	for $k=1,...,n,$ with $f_0\left(x\left(t\right)\right)=0$.\\
	This model can be written in a condensed way as
	\begin{subequations}
		\begin{align}
			\dot{x}\left(t\right)&=F\left(x\left(t\right)\right),\\
			\intertext{for $x=\left[i_1 \cdots i_n\hspace{0.1cm} s_1 \cdots s_n \right]^T$ and} F\left(x\left(t\right)\right)&=\left[F_1\left(x\left(t\right)\right) \cdots F_{2n}\left(x\left(t\right)\right)\right]^T,
		\end{align}
		
	\end{subequations}
	where 	\begin{align}
		\label{F}
		F_k\left(x\right)&=
		f_k\left(x\right),\nonumber\\
		F_{n+k}\left(x\right)&=\dfrac{1}{\beta_k\displaystyle\sum_{j=1}^n N_ji_j}\left(f_k\left(x\right)\left(T_k+\gamma_k-\alpha_2^k\right)\right.\\&\left.-\alpha_1^k i_k-T_kf_{k-1}\left(x\right)- \beta_k s_k\displaystyle\sum_{j=1}^n N_j f_j\left(x\right)\right)\nonumber
	\end{align} for $k=1,...,n$.
	\subsection{Stabilizing law}
	Moreover, in order to be effective, the feedback needs to ensure the eradication of I-individuals in the population.   \\
	\begin{thm}\label{stab_inf}\textbf{Stability of the I-individuals}\\
		Let the initial condition $x_0 \in R^{2n}_+$ be given. Assume that all roots $(-r_j^k)$ of the characteristic polynomial $P(s)$ associated with the closed-loop dynamics \eqref{CLdynamics} are in the open left half plane, i.e $Re(-r_j^k)<0$, by an appropriate choice of the control tuning parameters $\alpha_j^k>0$ for $j=1,2$ and $k=1,...,n$. \\
		Then the state feedback \eqref{control} implies the exponential convergence towards zero of the infected population $i_k(t)$ of NODE Model \eqref{ODE}, for $k=1,...,n$, as time tends to infinity.
	\end{thm} 
	\begin{proof}
	Since the closed-loop dynamics \eqref{CLdynamics} is a system of decoupled ODE's it can be written as
	\begin{align}
		\label{decoupled_ODE}
		\left\{
		\begin{array}{rll}
			\dot{\bar{x}}_{new}\left(t\right)&=&\hat{A}\bar{x}_{new}\left(t\right),\\
			y\left(t\right)&=&C\bar{x}_{new}\left(t\right)
		\end{array}
		\right.
	\end{align}
	with $\bar{x}_{new}=Px=\begin{pmatrix}
		\bar{i}_1&\bar{s}_1&\cdots& \bar{i}_n&\bar{s}_n
	\end{pmatrix}^T$ with $P$ a $2n \times 2n$ permutation matrix such that
	\begin{align*}
		P_{ij}=\left\{
		\begin{array}{rll}
			&1 & \text{ if } (i,j)=(2k+1,k+1) \text{ for } k=0,...,n-1 \\
			& 1 & \text{ if } (i,j)=(2k,n+k) \text{ for } k=1,...,n\\
			& 0 & \text{ otherwise},
		\end{array}
		\right.
	\end{align*}  $\hat{A}=blockdiag(\bar{A}_k)$, where $\bar{A}_k=\begin{pmatrix}
		0 & 1\\
		-\alpha_1^k & -\alpha_2^k
	\end{pmatrix}$ and $C=P\left[I_n \hspace{0.3cm} 0\right]$.\\
	Therefore, $\hat{A}$ is stable if all its eigenvalues are in the open left plane. However, the eigenvalues of $\hat{A}$ are those of the $\bar{A}_k$'s matrices. Moreover, those eigenvalues are the roots of the characteristic polynomial $P(s)=Det(sI-\bar{A}_k)=s^2-s\alpha_2^i+\alpha_1^i=(s+r_1^i)(s+r_2^i)$ with $\alpha_2^i=r_1^i+r_2^i$ and $\alpha_1^i=r_1^ir_2^i$. Therefore, the eigenvalues of the $\bar{A}_k$'s matrices are $-r_1^i$ and $-r_2^i$. Since they are of negative real part, then the control law exponentially stabilizes the model in normal form \eqref{normal_form}.	\\
	Therefore, $\bar{x}\left(t\right)$ exponentially converges asymptotically to zero. It follows that $\bar{i}_k\left(t\right)=i_k\left(t\right)$ converges to zero as time goes to infinity for $k=1,...,n$.\end{proof}\\ \\
	\begin{remark}\label{rem1} The control law is well-defined for $x \in \mathcal{D}$. However, since the aim is to eradicate the disease from the population, the infected population goes to zero as time tends to infinity. This implies that $\displaystyle\sum_{j=1}^n N_ji_j=\displaystyle\int_0^L I\left(t,a\right)da$ tends to zero. Therefore, as explained in \cite{Quesada}, we introduced a "switch-off" vaccination law, based on the fact that the disease is considered as being eradicated from the population when the infected population is greater than zero but small enough (for instance when there is numerically less than one individual in the population but more than zero). Therefore, we defined a threshold such that
	$0<\displaystyle\int_0^L I\left(t,a\right)da<\delta<1$. Therefore, in a practical situation, we use
	\begin{equation}\label{switch}u_s\left(t,a\right)=\left\{\begin{array}{l}
			u\left(t,a\right) \text{for }t\leq t^{\star},\\[0.4cm]
			0 \text{ for }t>t^{\star}
		\end{array}\right.
	\end{equation}
	where $$t^{\star}=min\left\{t\in\mathbb{R}^+ \vert \displaystyle\int_0^L I\left(t,a\right)da<\delta \text{ for }0<\delta<1\right\}.$$
\end{remark}
	\subsection{Positivity analysis}
	Another condition for the feedback design is that the feedback law  has to keep the positivity of the variables in the model (more precisely it has to keep them between $0$ and $1$) in order to have a physical meaning. Inspired by \cite{Tudor}, we highlight the following positivity condition.
	\begin{thm}\label{pos_cdt}	If $\theta_k(t)\geq 0$ for $k=1,...,n$, 
		then the set $B=\left\{\left(i_1,...,i_n, s_1,...,s_n : i_k \geq 0, s_k\geq 0, s_k+i_k\leq 1 \right.\right.$ $\left.\left. \text{for } k=1,...,n\right)\right\}$ is positively invariant for the ODE model \eqref{ODE}.
		\end{thm}
\begin{proof}
	Let $t_0$ be the smallest $t$ such that $x(t_0) \in \delta B$, where  $\delta B=\left\{\left(i_1,...,i_n, s_1,...,s_n : i_k = 0 \text{ or } s_k=0 \text{ or } s_k+i_k= 1\right)\right\}$.
	Assume that $s_m(t_0)=0$ ($i_m(t_0)=0$) for some $m$. By definition of $t_0$, all the other state components are such that $s_k(t_0), i_k(t_0)\geq 0$ for $k=1,...,m-1,m+1,...n$. Therefore, using equations \eqref{ODE},
	$\dfrac{ds_m\left(t_0\right)}{dt}\geq 0$ and 	$\dfrac{ds_m\left(t_0\right)}{dt}\geq 0$. On the other hand, if $i_m(t_0)+s_m(t_0)=1$ for some $m$, then $s_k(t_0), i_k(t_0)\geq 0$ and $s_k(t_0)+i_k(t_0)\leq 1$ for $k=1,...,m-1,m+1,...n$. By equations \eqref{ODE}, it follows that
	\begin{align*}
		\dfrac{d(s_m+im)(t_0)}{dt}&=T_m\left(s_{m-1}(t_0)+i_{m-1}(t_0)-1\right)\\&-\gamma_mi_m(t_0)-\theta_m(t_0)s_m(t_0)\\&\leq 0
			\end{align*}
			
	\end{proof}\\ \\
Observe that positivity is no longer guaranteed for all inputs on the system when working with the discretized model. However, assuming that $\theta_k(t)$, $k=1,...,n$ is greater than zero is not restrictive since, in order to have a physical meaning, this quantity needs to be positive.
	\subsection{Numerical Simulations}\label{NS2}
	Numerical simulations are performed to show that appropriate choices of parameters can guarantee the eradication of I-individuals. \\ In the simulations, parameters are taken from \cite{Okuwa} and described in Section \ref{NS} but are adapted to the ODE case as mentioned in Section \ref{Section_NODE}. In this case 100 classes of ages are considered. Moreover, the design control parameters are set to $r_1^k=200$ and $r_2^k=80$ for $k=1...100$. Simulations are stopped when convergence is reached with a tolerance of $10^{-8}.$ The code are performed using ODE45 function in \textsc{Matlab}.\\
	Remind that to have a physical meaning, the vaccination needs to be positive. Therefore, based on results found in numerical
	simulations, a new control law $u_{ps}
	(t)$ is designed where, for
	$k= 1,...,n,$
	\begin{align}
		\label{u_s}
		u_{ps,k}\left(t\right)=\left\{
		\begin{array}{ll}
			0, & \text{ if } u_{s,k}\left(t\right)<0,\\
			u_{s,k}\left(t\right)&\text{ otherwise } 	
		\end{array}
		\right.
	\end{align}
which is based on the control-law \eqref{switch} defined in Remark \ref{rem1}.
	Using this control law in numerical simulations shows that the stability of the system is conserved.
	\begin{figure}
		\begin{center}
			\includegraphics[height=5cm]{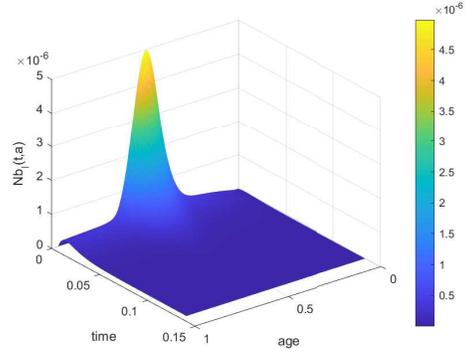}    % The printed column  
			\caption{Dynamics of I-individuals from NODE Model \eqref{ODE} with vaccination }  % width is 8.4 cm.
			\label{fig_i_ac}                                 % Size the figures 
		\end{center}
	\end{figure}
	\begin{figure} 
		\begin{center}
			\includegraphics[height=5cm]{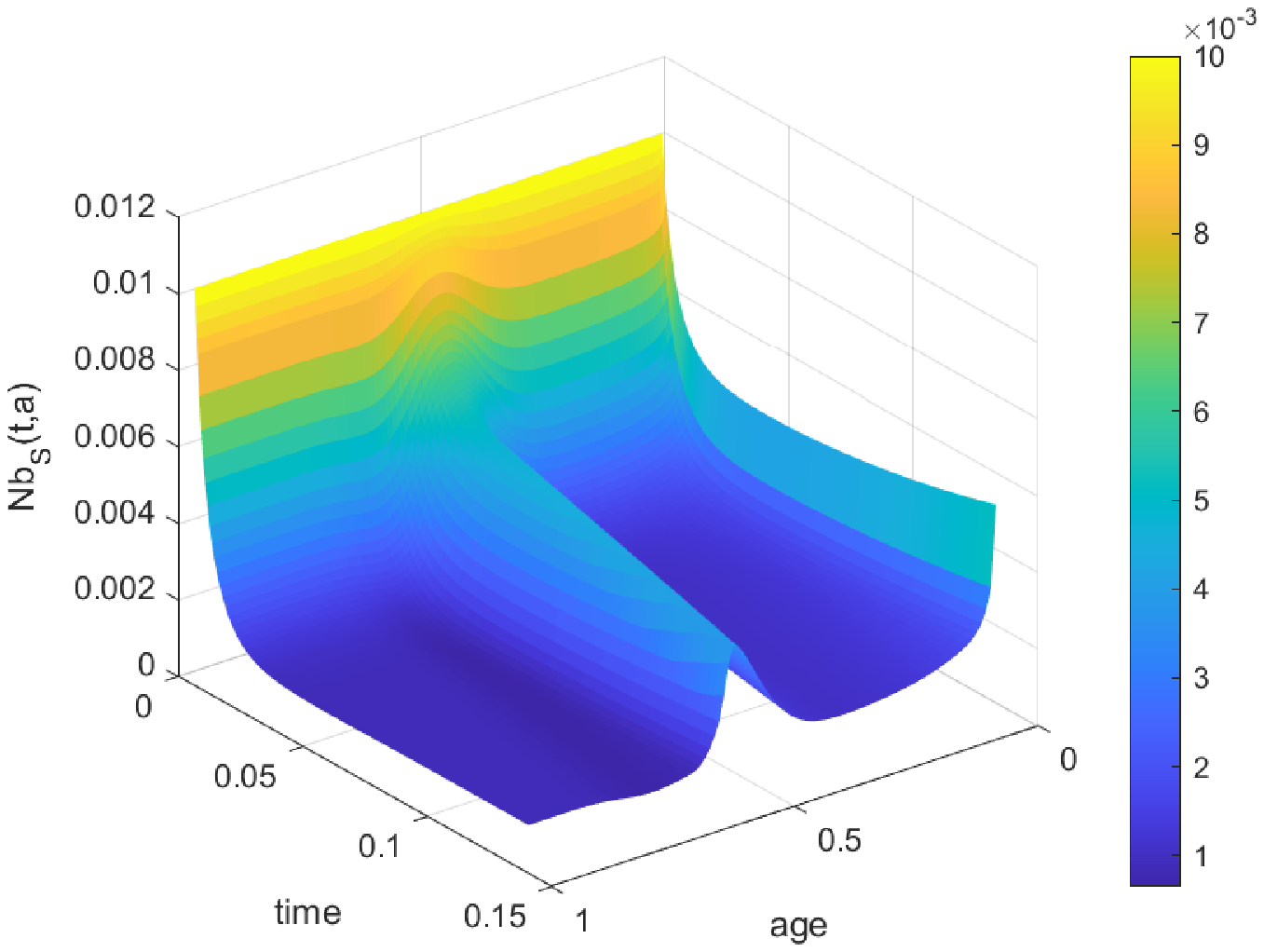}  
			\caption{Dynamics of S-individuals from NODE Model \eqref{ODE} with vaccination }  % width is 8.4 cm.
			\label{fig_s_ac}                                 % Size the figures 
		\end{center}                               
	\end{figure}
	\begin{figure}
		\begin{center}
			\includegraphics[height=5cm]{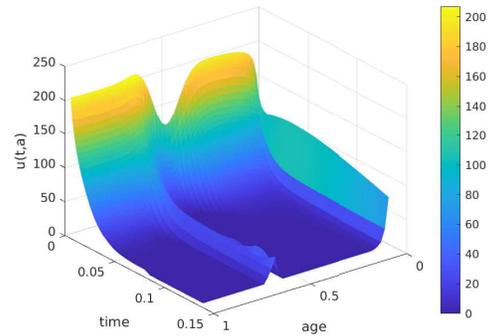}    % The printed column  
			\caption{Dynamics of vaccination }  % width is 8.4 cm.
			\label{fig_vac}                          
		\end{center}                                 
	\end{figure}
	\noindent Comparing Figure \ref{fig3} from Section \ref{NS} and Figure \ref{fig_i_ac}, we can observe that the I-individuals converge to zero when the control is applied. Moreover, the I-individuals remains positive and are smaller than 1. This is also the case for the number of S-individuals observed in Figure \ref{fig_s_ac}. Moreover, Figure \ref{fig_vac} suggests to vaccinate first individuals in the classes of age where the epidemic is absent and in a second time to vaccinate individuals from classes of age around the ages where individuals were initially infected.  
	\section{Closed-loop stabilization of PIDE model}\label{positive_PDE}
	The aim of this section is to design a closed-loop stabilization law for the PIDE Model. Therefore, we extend the control law \eqref{control} found in the ODE case to the PIDE case.
	\subsection{Feedback design}
	To discover the feedback design for the Model \ref{PDE_SIR}, let define a new function \begin{equation}
		W\left(t,a\right)=\Theta\left(t,a\right)S\left(t,a\right).\label{w}
	\end{equation} By putting \begin{equation}
		w_k\left(t\right)=\dfrac{1}{N_k}\displaystyle\int_{a_{k-1}}^{a_k}W\left(t,a\right)da \label{w_k_rel}
	\end{equation} and discretizing Model \ref{PDE_SIR} with respect to the age, it can be shown that \begin{equation}
		w_k\left(t\right)=\theta_k\left(t\right)s_k\left(t\right).\label{w_k}	
	\end{equation}
	Using relation \eqref{w_k_rel} and the limit of the mean value theorem for integral and assuming that $a\in\left[a_{k-1},a_k\right[$ for all $k\in\mathbb{N}$, implies that \begin{align*}
		& W\left(t,a\right)= \lim\limits_{\Delta_k\to 0} \dfrac{N_kw_k\left(t\right)}{\Delta_k}, \text{ with }\Delta_k=a_k-a_{k-1}\\
		\Leftrightarrow & \Theta\left(t,a\right) S\left(t,a\right)= \lim\limits_{\Delta_k\to 0} \dfrac{N_k\theta_k\left(t\right)s_k\left(t\right)}{\Delta_k}\\
		\Leftrightarrow & \Theta\left(t,a\right) = \dfrac{1}{S\left(t,a\right)}\lim\limits_{\Delta_k\to 0} \dfrac{N_k\theta_k\left(t\right)s_k\left(t\right)}{\Delta_k}
	\end{align*}
	Using relations of Section \ref{Section_NODE}, definition of the state feedback in finite dimension \eqref{control} and the definition of derivative in terms of limits give the following candidate for a nonlinear state feedback control law continuous in age
	\begin{align}
		\Theta\left(t,a\right)= &  \hspace{0.1cm}\tilde{\alpha}_2\left(a\right)+\displaystyle\int_0^L \beta\left(a\right)S\left(t,a\right)da-2\mu\left(a\right)-\gamma\left(a\right)\nonumber\\& -\beta\left(a\right)\displaystyle\int_0^LI\left(t,b\right)db\nonumber\\&-\dfrac{\displaystyle\int_0^L\left(\mu\left(a\right)+\gamma\left(a\right)\right)I\left(t,a\right)da}{\displaystyle\int_0^LI\left(t,b\right)db}\nonumber\\& +\dfrac{I\left(t,a\right)}{\beta\left(a\right)S\left(t,a\right)\displaystyle\int_0^LI\left(t,b\right)db}\Big(\tilde{\alpha}_1\left(a\right)\nonumber\\&+\left(\mu\left(a\right)+\gamma\left(a\right)\right)\left(\mu\left(a\right)+\gamma\left(a\right)-\tilde{\alpha}_2\left(a\right)\right)\Big)
		\label{Theta}
	\end{align} 
	Note that in this feedback, we divide by $\displaystyle\int_0^LI\left(t,b\right)db$. However, since the eradication of the epidemic is wanted, this quantity tends to zero. Therefore, in practice we use a switch-vaccination law as it was done in Remark \ref{rem1}. The control is defined as the one in \eqref{switch}. %With this feedback, the closed loop PDEs model becomes
	%\begin{align}
	%	\left(\partial_t \right.&\left.+ \partial_a\right) S\left(t,a\right)= S\left(t,a\right)\Big[-\tilde{\alpha}_2+\left(\mu\left(a\right)+\gamma\left(a\right)\right)\nonumber\\&\left.-\displaystyle\int_0^L \beta\left(a\right)S\left(t,a\right) da +\dfrac{\displaystyle\int_0^L  \left(\mu\left(a\right)+\gamma\left(a\right)\right)I\left(t,a\right)da}{\displaystyle\int_0^L I\left(t,b\right)db}\right]\nonumber\\  &\hspace{-0.24cm}+I\left(t,a\right)\left[\dfrac{-\tilde{\alpha}_1+\left(\mu\left(a\right)+\gamma\left(a\right)\right)\left(\tilde{\alpha}_2-\left(\mu\left(a\right)+\gamma\left(a\right)\right)\right)}{\beta\left(a\right)\displaystyle\int_0^L I\left(t,b\right)db}\right],\nonumber \\[0.5cm]
	%	\left(\partial_t \right.&\left.+ \partial_a \right)I\left(t,a\right)= -\left(\gamma\left(a\right)+\mu\left(a\right)\right)I\left(t,a\right)\label{CLPDE}\\&+\beta\left(a\right)S\left(t,a\right)\displaystyle\int_0^L I(t,b) db,\nonumber\\[0.3cm]
	%	&\hspace{1.2cm}y\left(t\right)=I\left(t,a\right)\nonumber
	%\end{align}
	%
	%		under boundary conditions $
	%				S\left(t,0\right)= B,$ and $
	%				I\left(t,0\right)=0$
	%		and initial conditions $
	%				S\left(0,a\right)=S_0\left(a\right)
	%				I\left(0,a\right)=I_0\left(a\right).$ Note that since $R\left(t,a\right)+S\left(t,a\right)+I\left(t,a\right)=P\left(t,a\right):=P\left(a\right)$ with $P\left(a\right)$ given by $B\exp\left(-\displaystyle\int_0^a \mu\left(\eta\right)d\eta\right)$ there is no need to consider the equation for $R\left(t,a\right)$ since it equals to $P\left(a\right)-S\left(t,a\right)-I\left(t,a\right)$. 
	\subsection{Stabilizing law}
	The aim of this section is to show that the feedback law \eqref{Theta} stabilizes the PIDE model.\\ \\
	Inspired by Isidori's theory \cite{Isidori}, the following nonlinear coordinate changes is made, 
	\begin{align}
		\bar{I}\left(t,a\right)&=I\left(t,a\right),\nonumber\\
		\bar{S}\left(t,a\right)&=-\left(\gamma\left(a\right)+\mu\left(a\right)\right)I\left(t,a\right)\label{chgmt_coord}\\&+\beta\left(a\right)S\left(t,a\right)\displaystyle\int_0^L I\left(t,b\right)db.\nonumber
	\end{align}
	In this formulation the open-loop Model \eqref{PDE_SIR} becomes
	\begin{align}
		\left(\partial_t + \partial_a\right) &\bar{I}\left(t,a\right)=\bar{S}\left(t,a\right),\nonumber\\
		\left(\partial_t + \partial_a\right) &\bar{S}\left(t,a\right)=-\left(\gamma\left(a\right)+\mu\left(a\right)\right)\bar{S}\left(t,a\right)\nonumber\\&-\partial_a\left(\gamma\left(a\right)+\mu\left(a\right)\right)\bar{I}\left(t,a\right)\nonumber\\&+\left[\bar{S}\left(t,a\right)+\left(\gamma\left(a\right)+\mu\left(a\right)\right)\bar{I}\left(t,a\right)\right]\big[-\Theta\left(t,a\right)\nonumber\\&-\mu\left(a\right)-\beta\left(a\right)\displaystyle\int_0^L\bar{I}\left(t,b\right)db+\dfrac{\displaystyle\int_0^L\bar{S}\left(t,b\right)db}{\displaystyle\int_0^L\bar{I}\left(t,b\right)db}\nonumber\\&\left.+\dfrac{\partial_a\beta\left(a\right)}{\beta\left(a\right)}\right]\label{open_loop}
	\end{align}
	under non-homogeneous boundary conditions
	\begin{align}
		\bar{I}\left(t,0\right)&=0,\nonumber\\
		\bar{S}\left(t,0\right)&= \beta\left(0\right)B\displaystyle\int_0^L \bar{I}\left(t,b\right)db\label{NHBC}
	\end{align}
	and initial conditions 
	\begin{align}
		\bar{I}\left(0,a\right)&=I_0\left(a\right),\nonumber\\
		\bar{S}\left(0,a\right)&=\bar{S}_0\left(a\right).\label{IC}
	\end{align} Moreover, the vaccination law \eqref{Theta} rewrites 
	\begin{align*}
		\Theta\left(t,a\right)&=\tilde{\alpha}_2\left(a\right)+\dfrac{\displaystyle\int_0^L\bar{S}\left(t,b\right)db}{\displaystyle\int_0^L\bar{I}\left(t,b\right)db}-2\mu\left(a\right)-\gamma\left(a\right)\nonumber\\&-\beta\left(a\right)\displaystyle\int_0^L\bar{I}\left(t,b\right)db\nonumber\\&+\dfrac{\bar{I}\left(t,a\right)}{\bar{S}\left(t,a\right)+\left(\gamma\left(a\right)+\mu\left(a\right)\right)\bar{I}\left(t,a\right)}\big(\tilde{\alpha}_1\left(a\right)\nonumber\\&+\left(\mu\left(a\right)+\gamma\left(a\right)\right)\left(\mu\left(a\right)+\gamma\left(a\right)-\tilde{\alpha}_2\left(a\right)\right)\big)
	\end{align*}
	Therefore, the closed-loop system is given by
	\begin{align}
		\left(\partial_t + \partial_a\right) \bar{I}\left(t,a\right)&=\bar{S}\left(t,a\right),\nonumber\\
		\left(\partial_t + \partial_a\right) \bar{S}\left(t,a\right)&=\bar{I}\left(t,a\right)\left[-\tilde{\alpha}_1\left(a\right)+g\left(a\right)\right]\nonumber\\&\phantom{=}+\bar{S}\left(t,a\right)\left[-\tilde{\alpha}_2\left(a\right)+h\left(a\right)\right]\label{CLNV}
	\end{align}
	where we note 
	\begin{align}
		g\left(a\right)&=-\beta\left(a\right)\dfrac{d}{da}\left(\dfrac{\gamma\left(a\right)+\mu\left(a\right)}{\beta\left(a\right)}\right),\label{ka}\\
		h\left(a\right)&=\dfrac{1}{\beta\left(a\right)}\dfrac{d}{da}\beta\left(a\right)\label{ha}
	\end{align} with boundary conditions \eqref{NHBC} and initial conditions \eqref{IC}.
	The design parameters are denoted by $\tilde{\alpha}_1\left(a\right)$ and $\tilde{\alpha}_2\left(a\right)$. They can be chosen appropriately to have the stability and positivity of the system. Others parameters ($\beta, \mu, \gamma$) are given for a chosen model. \\
	Similarly to the results of Isidori's theory for finite-dimensional system \cite{Isidori}, we find a feedback that linearizes the open-loop system \eqref{open_loop} thanks to the appropriate coordinates change \eqref{chgmt_coord}.\\ \\ It remains to show that the feedback \eqref{Theta} also stabilizes system \eqref{PDE_SIR}.
	In the following it is shown that the closed-loop system \eqref{CLNV} is stable which implies the asymptotic convergence to zero of the infected population.
	Hereafter we denote $$G\left(a\right)=-\tilde{\alpha}_1\left(a\right)+g\left(a\right) \text{ and } H\left(a\right)=-\tilde{\alpha}_2\left(a\right)+h\left(a\right). $$ With those notations, the state-space formulation of system \eqref{CLNV} is given by 
	\begin{align}
		\dot{\bar{x}}&=\bar{\mathcal{A}}\bar{x}\nonumber\\
		\bar{x}\left(0\right)&=\bar{x}_{B0}\label{SSF}
	\end{align}
	with $\bar{x}=\left(\bar{I}, \bar{S}\right)^T,$ $\bar{\mathcal{A}}=\begin{pmatrix}
		-\dfrac{d\cdot}{da} & I\\
		G\left(a\right)I & \phantom{iii}-\dfrac{d\cdot}{da}+H\left(a\right)I
	\end{pmatrix}$ with $\mathcal{D}\left(\bar{\mathcal{A}}\right)=\left\{\bar{x}\in L^1\left(0,L\right)\times\left(0,L\right),\bar{x}, \dfrac{d\bar{x}}{da} AC\left[0,L\right],\right.\\\bar{x}\left(0\right)=\bar{x}_{B0}\Big\}$  and $\bar{x}_{B0}=\begin{pmatrix}
		0\\ \beta\left(0\right) B\displaystyle\int_0^L\bar{I}\left(t,b\right)db
	\end{pmatrix}.$\\ \\
	Using Fattorini's approach \cite{fattorini_1968} on boundary control systems and generalizing the results in \cite[Chap. 10]{tucsnak_weiss_2009} to Banach spaces (see Appendix \ref{BC_L1}), system \eqref{SSF} can be rewritten as follows:
	\begin{align}
		\dot{\bar{x}}&=\bar{\mathcal{A}}_0\bar{x}+\tilde{B}u\nonumber\\
		\bar{x}\left(0\right)&=\left(0, 0\right)^T\label{SSF_BC}
	\end{align}
	where $\bar{\mathcal{A}}_0=\bar{\mathcal{A}}$ with $\mathcal{D}\left(\bar{\mathcal{A}}_0\right)=\left\{\bar{x}\in L^1\left(0,L\right)\times L^1\left(0,L\right):\right.\\\left.\bar{x}, \dfrac{d\bar{x}}{da} AC\left[0,L\right], \bar{x}_0=\left(0,0\right)^T \right\}.$ Moreover  $\tilde{B}=\begin{pmatrix}
		0\\\delta_0  
	\end{pmatrix}$ and $u=\beta\left(0\right)
	B\displaystyle\int_0^L\bar{I}\left(t,b\right)db.$ \\Equivalently we have
	\begin{align}
		\dot{\bar{x}}&=\left(\bar{\mathcal{A}}_0+\bar{D}\right)\bar{x}\nonumber\\
		\bar{x}\left(0\right)&=\left(0, 0\right)^T\label{SSF_BC_2}
	\end{align}
	where $\bar{D}=\begin{pmatrix}
		0 & 0\\\delta_0\beta\left(0\right)B\displaystyle\int_0^L\cdot \hspace{0.2cm} db &0  
	\end{pmatrix}$. \\ \\
	So that $\bar{D}$ is bounded, we use an approximation of the Dirac delta $\delta_0$. Remark that this approximation allow us to deal with a more realistic model since there is no sense to vaccinate instantaneously at birth. Let define $d_k\left(a\right)$ a term of a Dirac sequence which satisfies the properties developed in \cite[Chap. 2, Sect. 3, Lemma 2.3.4]{pritchard_2010} with $\infty$ replaced with $L$.
	Therefore, \eqref{SSF_BC} becomes 
	\begin{align}
		\dot{\bar{x}}_k&=\left(\bar{\mathcal{A}}_0+\bar{D}_k\right)\bar{x}_k\nonumber\\
		\bar{x}_k\left(0\right)&=\left(0, 0\right)^T\label{statespace}
	\end{align}
	with $\bar{D}_k=\begin{pmatrix}
		0 & 0\\ d_k\left(a\right)\beta\left(0\right)B\displaystyle\int_0^L\cdot \hspace{0.2cm} db\hspace{0.2cm} & 0
	\end{pmatrix}$.\\
	The choice of the term of the Dirac sequence does not have any impact in the bound of $B$ since its integral equals one.\\ Therefore, in the following, an approximation of model \eqref{SSF_BC_2}, where the unbounded operator $D$ is replaced by the bounded operator $\bar{D}_k$, is used in order to perform the stability and analysis.
	\begin{lemma}\label{GI}
		$\bar{\mathcal{A}}_0$ is the infinitesimal generator of a $C_0-$semigroup $\bar{U}\left(t\right)$.
	\end{lemma}
	\begin{proof}
	As in \cite{aksikas_winkin_dochain_2007} we perform a similarity transformation to have an equivalent state space description as \eqref{statespace} with triangular infinitesimal generator. Therefore, $J=\begin{pmatrix}
		I & 0\\ \kappa\left(a\right)I & I
	\end{pmatrix}$ is chosen where $\kappa\left(a\right)$ is a bounded $C_1$ solution (i.e $\exists$ $K>0$ such that $\kappa\left(a\right))\leq K$ $\forall$ $a\in\left[0,L\right]$) of the equation $$\dfrac{d\kappa}{da}+\kappa^2\left(a\right)-H\left(a\right)\kappa\left(a\right)+G\left(a\right)=0.$$ Note that the existence of such solution is established in \cite{aksikas_winkin_dochain_2007}. Applying this transformation to the operator $\bar{\mathcal{A}}_0$, i.e $J^{-1}\bar{\mathcal{A}}_0J$, gives $$\tilde{\mathcal{A}}=\begin{pmatrix}
		-\dfrac{d\cdot}{da}+G\left(a\right)\kappa\left(a\right)I & I\\ 0 & -\dfrac{d\cdot}{da}+\left(H\left(a\right)-\kappa\left(a\right)\right)I
	\end{pmatrix} $$ where $\mathcal{D}\left(\tilde{\mathcal{A}}\right)=\mathcal{D}\left(\bar{\mathcal{A}}_0\right)$. By a result developed in \cite[Chap. 5, Sect. 3, Lemma 5.3.2]{curtain_zwart_2020} $\tilde{\mathcal{A}}$ is the infinitesimal generator of a $C_0-$semigroup \begin{align*}
		\left(\tilde{U}\left(t\right)\right)_{t\geq 0}=\begin{pmatrix}
			\tilde{U}_1\left(t\right) & \tilde{U}_{12}\left(t\right)\\
			0 & \tilde{U}_2\left(t\right)
		\end{pmatrix}
	\end{align*}
	where 
	\begin{align}
		\left(\tilde{U}_{12}\left(t\right)x\right)\left(a\right)=\displaystyle\int_0^t\tilde{U}_1\left(t-s\right)\tilde{\mathcal{A}}_{12}\tilde{U}_2\left(s\right)x\left(a\right)ds.\label{SG}
	\end{align}
	Using Lemma 4.5 in \cite{schumacher_1981}, we conclude that $\bar{\mathcal{A}}_0$ is the infinitesimal generator of a $C_0-$semigroup $$\left(\bar{U}\left(t\right)\right)_{t\geq 0}=\left(J\tilde{U}\left(t\right)J^{-1}\right)_{t\geq 0}.$$
		\end{proof}
	\begin{thm}\textbf{Semigroup generation}\\
		$\bar{\mathcal{A}}_0+\bar{D}_k$ is the infinitesimal generator of a $C_0-$semigroup $\left(\bar{T}\left(t\right)\right)_{t\geq 0}$.
	\end{thm}
\begin{proof}
	The linear operator $\bar{D}_k$ is bounded on $L^1\left(0,L\right)\times L^1\left(0,L\right)$ with $\Vert \bar{D}_k \Vert:=\Vert \bar{D}_k \Vert_{\mathcal{L}\left(L^1\times L^1\right)}\leq \beta\left(0\right)B$). Moreover, by Lemma \ref{GI}, $\bar{\mathcal{A}}_0$ is the infinitesimal generator of a $C_0-$semigroup. Therefore, we may apply the bounded perturbation theorem developed in \cite[Chap. 3, Sect. 1]{Engel_Nagel} which concludes the proof. \end{proof}
	\begin{lemma}\label{stability}
		$\tilde{\mathcal{A}}$ is the infinitesimal generator of an exponentially stable $C_0-$semigroug $(\tilde{U}(t))_{t\geq 0}$ with growth constant $$\tilde{\omega}_0(\tilde{U})<-(c_2+K)<0 $$provided that the gain functions $\tilde{\alpha}_1(a)$ and $\tilde{\alpha}_2(a)$ be chosen such that \begin{align}
			\tilde{\alpha}_1\left(a\right) &\text{ such that } g\left(a\right)+c_1>\tilde{\alpha}_1\left(a\right)\label{rel_1}\\
			\tilde{\alpha}_2\left(a\right) &\text{ such that } h\left(a\right)+c_2>\tilde{\alpha}_2\left(a\right)\label{rel_2}
		\end{align} 
		for all $a\in\left[0,L\right]$ with $0\leq c_2\leq K\left(c_1-1\right)$ and $1\leq c_1$.
	\end{lemma}
	\begin{proof}
	By the method of characteristics, we show that \begin{align*}
		\left(\tilde{U}_1\left(t\right)x_{01}\right)\left(a\right)&=\left\{
		\begin{array}{ll}
			x_{01}\left(a-t\right)E_{G\kappa}\left(a-t,a\right), & \text{ if } t\leq a, \\
			0 &  \text{ if } t > a,
		\end{array}
		\right.\\
		\left(\tilde{U}_2\left(t\right)x_{02}\right)\left(a\right)&=\left\{
		\begin{array}{ll}
			x_{02}\left(a-t\right)E_{H-\kappa}\left(a-t,a\right), & \text{ if } t\leq a, \\
			0 &  \text{ if } t > a
		\end{array}
		\right.
	\end{align*}
	with $E_f\left(x,y\right)=\exp\left(\displaystyle\int_x^yf\left(\eta\right)d\eta\right).$
	Moreover, using relation \eqref{SG}, we find that
	\begin{align*}
		\left(\tilde{U}_{12}\left(t\right)x_{0}\right)\left(a\right)&=\left\{
		\begin{array}{ll}
			x_{0}\left(a-t\right)\displaystyle\int_0^t E_{H-\kappa}\left(a-t,\right.&\left.\hspace{-0.2cm}a-t+s\right)\\\hspace{1cm}E_{G\kappa}\left(a-t,a-s\right)ds & \text{ if } t\leq a, \\
			0 &  \text{ if } t > a,
		\end{array}
		\right.
	\end{align*}
	Therefore, in view of the trajectories, $\tilde{\mathcal{A}}$ is the infinitesimal generator of a stable $C_0-$semigroup with growth bound $\omega_0$ equals $-\infty$.
	Let $x_0=\left(x_{01}, x_{02}\right)^T$. Since $\tilde{\mathcal{A}}$ is stable, we know that $\exists M_c>0$ and $c>0$ such that 
	$$
	\Vert \tilde{U}\left(t\right)x_0\Vert\leq M_ce^{-ct}\Vert x_0\Vert.$$ In the following part of the proof, we identify $c$ and $M_c$. \begin{align*}
		\Vert \tilde{U}\left(t\right)x_0\Vert&=\displaystyle\int_0^L \vert \left(\tilde{U}_1\left(t\right)x_{01}\right)\left(a\right)+\left(\tilde{U}_{12}\left(t\right)x_{02}\right)\left(a\right)\vert da\\&+\displaystyle\int_0^L\vert \left(\tilde{U}_2\left(t\right)x_{02}\right)\left(a\right)\vert da\\
		&\leq \displaystyle\int_0^L \vert x_{01}\left(\eta\right)\vert E_{G\kappa}\left(\eta,\eta+t\right) d\eta \\ &+\displaystyle\int_0^L \vert x_{02}\left(\eta\right)\vert \Big(E_{H+\kappa}\left(\eta,\eta+t\right)\\&\left.+\displaystyle\int_0^t E_{H+\kappa}\left(\eta,\eta+s\right)E_{G\kappa}\left(\eta,\eta+t-s\right)ds\right)  d\eta
	\end{align*}
	Using relations \eqref{rel_1} and \eqref{rel_2}, we can show that 
	$$E_{G\kappa}\left(\eta,\eta+t\right)\leq e^{-K c_1 t} \text{ and } E_{H+\kappa}\left(\eta,\eta+t\right)\leq e^{-\left(K +c_2\right)t}.$$
	Thus,
	\begin{align*}
		\Vert \tilde{U}\left(t\right)x_0\Vert\leq& e^{-K c_1 t}\Vert x_{01}\Vert_1\\&+\left(\dfrac{e^{-\left(K +c_2\right)t}-e^{-K c_1 t}}{K c_1-\left(K +c_2\right)}+e^{-\left(K +c_2\right)t}\right)\Vert x_{02}\Vert_1\\
		\leq & \left(1+K\left(c_1-1\right)-c_2\right)e^{-\left(c_2+K\right)t}\Vert x_0 \Vert.
	\end{align*}	\end{proof}
	\begin{lemma}\label{stab}\textbf{Stability of $\bar{\mathcal{A}}_0+\bar{D}_k$}\\
		$\bar{\mathcal{A}}_0+\bar{D}_k$ is the infinitesimal generator of an exponentially stable $C_0-$semigroup  $\left(\tilde{T}\left(t\right)\right)_{t\geq 0}$ with growth bound $$\omega_0(\tilde{T})<-\left(c_2+K\right)+\left(1+K\left(c_1-1\right)-c_2\right)\Vert\bar{D}_k\Vert<0$$ if $c_1$ and $c_2$ are chosen such that \begin{align}
			c_1&>max\left\{1,\underset{a\in\left[0,L\right]}{sup}\left(\tilde{\alpha}_1\left(a\right)-g\left(a\right)\right),\dfrac{\beta_0B}{K}\right\},\label{rel_4}\\
			c_2&>max\left\{0,\dfrac{\beta_0B\left(1+K c_1\right)}{1+\beta_0B}-K,\underset{a\in\left[0,L\right]}{sup}\left(\tilde{\alpha}_2\left(a\right)-h\left(a\right)\right)\right\},\label{rel_5}\\
			c_2&\leq K\left(c_1-1\right).\label{rel_6}
		\end{align}
	\end{lemma}
	\begin{proof}
	In order to use the invariance of stability under system equivalence, we apply the transformation $J$ to the operator $\bar{\mathcal{A}}_0+\bar{D}_k$, which gives the operator $\tilde{\mathcal{A}}+\bar{D}_k$. By the bounded perturbation theorem from \cite[Chap. 3, Sect. 1]{Engel_Nagel} we know that $\tilde{\mathcal{A}}+\bar{D}_k$ is the infinitesimal generator of a $C_0-$semigroup $\left(\tilde{T}\left(t\right)\right)_{t\geq 0}$ satisfying \begin{align*}
		\Vert \tilde{T}\left(t\right)\Vert\leq& M_c e^{\left(-c+M_c\Vert \bar{D}_k\Vert\right)t}\\
		\Leftrightarrow \Vert \tilde{T}\left(t\right)\Vert\leq& \left(1+K\left(c_1-1\right)-c_2\right)\\&e^{\left(-\left(c_2+K\right)+\left(1+K\left(c_1-1\right)-c_2\right)\Vert\bar{D}_k\Vert\right)t}
	\end{align*} by Lemma \ref{stability} where assumptions \eqref{rel_1} and \eqref{rel_2} are included in \eqref{rel_4} to \eqref{rel_6}.
	Moreover, using $\Vert \bar{D}_k\Vert\leq \beta_0B$ and relations \eqref{rel_5} and \eqref{rel_6} implies that $\left(-c+M_c\Vert \bar{D}_k\Vert\right)<0$. Therefore $\tilde{\mathcal{A}}+\bar{D}_k$ is stable involving the stability of $\bar{\mathcal{A}}_0+\bar{D}_k$.\\ Note that inequation \eqref{rel_4} implies the feasibility of relations \eqref{rel_5} and \eqref{rel_6}.	\end{proof}
	\begin{thm}\label{stab_inf_PDE}\textbf{Stability of the I-individuals}\\
		Let $x_{0_k}=\left[I_{0_k},S_{0_k}\right]^T\in L^1\left(0,L\right)\times L^1\left(0,L\right)$. Assume that we choose $c_1$, $c_2$, $\tilde{\alpha}_1\left(a\right)$ and $\tilde{\alpha}_2\left(a\right)$ such that conditions \eqref{rel_4} to \eqref{rel_6} are satisfied. Then, the state feedback \eqref{Theta} implies the exponential asymptotic convergence to zero of the infected population $I_k\left(t,a\right)$ as time tends to infinity:\\ $\Vert I_k(t,\cdot) \Vert_1 \to 0$ as time goes to infinity. 	
	\end{thm}
	\begin{proof}
	Since by Lemma \ref{stab}, $\bar{x}_k\left(t\right)$ exponentially converges to zero. Therefore, by relation \eqref{chgmt_coord}, $\bar{I}_k\left(t,a\right)=I_k\left(t,a\right)$ exponentially tends to zero.
		
\end{proof}
	\begin{remark}
	In view of this analysis, we conjecture that $I(t,a)$ asymptotically exponentially converges to zero.\\
	Intuitively, we have that $\bar{I}_k\left(t,a\right)$ and $\bar{S}_k\left(t,a\right)$ tend to $\bar{I}\left(t,a\right)$ and $\bar{S}\left(t,a\right)$ as $k$ tends to infinity. This idea can be shown by studying the limits of the error's dynamics  $E\left(t,a\right)=\left(\bar{I}_k\left(t,a\right)-\bar{I}\left(t,a\right), \bar{S}_k\left(t,a\right)-\bar{S}\left(t,a\right) \right)^T$ which is given, using \eqref{SSF_BC_2} and \eqref{statespace} by \begin{align}
		\dot{E}&=\bar{D}\bar{x}+\bar{D}_k\bar{x}_k\nonumber\\
		E\left(0\right)&=\left(0, 0\right)^T.
	\end{align}  This tends to $\dot{E}=\bar{D}E$ with $E\left(0\right)=\left(0, 0\right)^T$ as $k$ tends to infinity. The solution of this differential equation is $E=0$. Therefore we can hypothesized that $\bar{I}_k\left(t,a\right)$ and $\bar{S}_k\left(t,a\right)$ tend to $\bar{I}\left(t,a\right)$ and $\bar{S}\left(t,a\right)$ as $k$ tends to infinity.\\
	Note that this intuition is corroborated by numerical simulations.	\end{remark}  Finally, We can observe that the state trajectories remain positive for the closed-loop system when positive initial conditions are taken. This can be shown using similar arguments as the ones used for the well-posedness of the HNPIDE model. Specifically for the S-individuals, the methods of the characteristics gives,
		$$ s\left(a,t\right)=\left\{\begin{array}{l}
			\exp\left(-\displaystyle\int_0^a F\left(s(\eta,\eta+t-a),i(\eta,\eta+t-a)\right)\right. \\ \phantom{exp()}+\lambda\left(\eta,\eta+t-a\right)d\eta\Big) \hspace{0.2cm} \text{if }t>a,\\[0.5cm]
			s_0\left(a-t\right)\exp\left(-\displaystyle\int_0^t \lambda\left(\zeta+a-t,\zeta\right)d\zeta \right.\\+F\left(s(\zeta+a-t,\zeta),i(\zeta+a-t,\zeta)\right)\Big) \hspace{0.2cm} \text{if }t\leq a. \end{array}
		\right. $$ where $\lambda(t,a)=\beta(a)\displaystyle\int_0^L P(t,b)i(t,b)db$ and $\theta(t,a)=F\left(s(t,a),i(t,a)\right)$ since it is a state-feedback law.
	This implicit equation remains positive under positive inital conditions.
	\subsection{Design procedure and numerical simulations}
To perform numerical simulations the feedback gains need to be chosen the appropriately in order to ensure the disease eradication but also the positivity of vaccination in order to have physical meaning.\\ 
First, it can be noticed that disease eardication can be achieved regardless of the choice of the parameters. Indeed as it can be viewed in Lemma \ref{stab}, the choice of the design parameters only impact the convergence speed of the system. Therefore they can be tuned to achieved a desired stability margin.\\ In the following, some conditions on the design parameters are highlighted in order to ensure positivity of the vaccination law. 
\begin{thm}\textbf{Sufficient conditions for the positivity of the vaccination law}\\
Define  $$\nu = \underset{a\in[0,L]}{\sup} \mu(a); \Gamma =\underset{a\in[0,L]}{\sup} \gamma(a)$$
and $N$ the total population.
Taking 
\begin{align}
\tilde{\alpha}_2(a)&=3\nu+2\Gamma+\beta(a)N\label{alpha2}\\
\tilde{\alpha}_1(a)&=-\left(\mu(a)+\gamma(a)\right)\left(\mu(a)+\gamma(a)-\tilde{\alpha}_2\right)\label{alpha1}
\end{align}	
	yields the locally exponentially stable closed-loop system \eqref{CLNV}-\eqref{ha}, \eqref{chgmt_coord} with the positive vaccination law \eqref{Theta}.
\end{thm}
\begin{proof}
The vacination law \eqref{Theta} with definition \eqref{alpha1} rewrittes 
\begin{align*}
	\Theta\left(t,a\right)= &  \hspace{0.1cm}\tilde{\alpha}_2\left(a\right)+\displaystyle\int_0^L \beta\left(a\right)S\left(t,a\right)da-2\mu\left(a\right)-\gamma\left(a\right)\nonumber\\& -\beta\left(a\right)\displaystyle\int_0^LI\left(t,b\right)db\nonumber\\&-\dfrac{\displaystyle\int_0^L\left(\mu\left(a\right)+\gamma\left(a\right)\right)I\left(t,a\right)da}{\displaystyle\int_0^LI\left(t,b\right)db}\nonumber.
\end{align*}
Moreover, since $\nu \geq \mu(a)$ for all $a\in [0,L]$, $\Gamma \geq \gamma(a)$ for all $a\in[0,L]$ and $\displaystyle\int_0^L I(t,b)db\leq N$, we get the following estimate for the vaccination law, 
\begin{align*}
	\Theta(t,a) \geq&\hspace{0.1cm}\tilde{\alpha}_2\left(a\right)+\displaystyle\int_0^L \beta\left(a\right)S\left(t,a\right)da-2\nu-\Gamma\nonumber\\&-\beta(a)N-(\nu + \Gamma)\\
	&\geq \displaystyle\int_0^L \beta\left(a\right)S\left(t,a\right)da
\end{align*}
using definition \eqref{alpha2}. It follows that the vaccination law with those choice of parameters is positive.\end{proof}\\ \\
%In numerical simulations, the following algorithm is used. 
%	\begin{algo}
%		Computation of feedback gains \hrule 
%		\textbf{Inputs :} Parameters of Model \ref{PDE_SIR} :\\ $\mu\left(a\right)$, $\beta(a)$, $\gamma(a)$, initial conditions and boundary conditions.
%		
%		\textbf{Step 1.} Compute $g(a)$ and $h(a)$ as mentioned in \eqref{ka} and \eqref{ha}. \\
%		\textbf{Step 2.} Take randomly $c_1>1$ and some coefficients $0<coef_i<10$ for $i=1...6$. (It can be adapted according to the shape wanted for the function $\tilde{\alpha}_1$. )\\
%		\textbf{Step 3.} Define $$\tilde{\alpha}_1\left(a\right)=g(a)-coef_1-coef_2a-coef_3a^2.$$
%		\textbf{Step 4.} Choose \begin{align*}
%			\tilde{\alpha}_2\left(a\right)=&\beta\left(0\right)B\left(c_1-1\right)+h(a)-coef_4-coef_5a\\&-coef_6a^2.
%		\end{align*}
%		\textbf{Outputs :} Feedback gain functions such that the closed-loop system (Model \ref{PDE_SIR} with control law \eqref{Theta}) is stable and positive.\hrule
%	\end{algo}
%	Remark that those choices are based on the assumption that $K$ can be chosen greater than $\beta\left(0\right) B$, since it is an upper bound of $\Gamma$. It follows that it can be taken as big as we wish. \\ Remark that, as in the ODE case, the vaccination law needs to be positive in order to have a physical meaning. Therefore, in numerical simulations a switch control law is used,
%	\begin{align}
%		\label{vac}
%		\Theta_s\left(t,a\right)=\left\{
%		\begin{array}{ll}
%			\Theta\left(t,a\right), & \text{ if } \Theta\left(t,a\right)\geq 0,\\
%			0&\text{ otherwise. } 	
%		\end{array}
%		\right.
%	\end{align}
Therefore, in the PDE case, there is no need to use a switch vaccination law. Simulations are performed using the same parameters and tolerance defined in Sections \ref{NS} and \ref{NS2}.
	\begin{figure}
		\begin{center}
			\includegraphics[height=5cm]{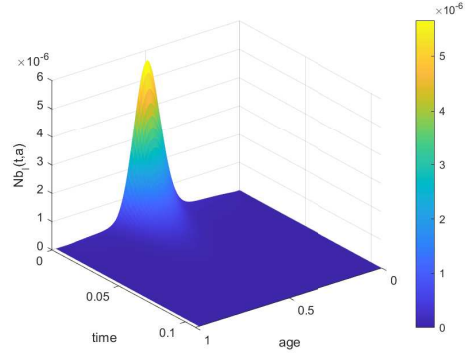}    % The printed column  
			\caption{Dynamics of I-individuals from PIDE Model with vaccination }  % width is 8.4 cm.
			\label{fig_i_pde}                                 % Size the figures 
		\end{center}
	\end{figure}
	\begin{figure} 
		\begin{center}
			\includegraphics[height=5cm]{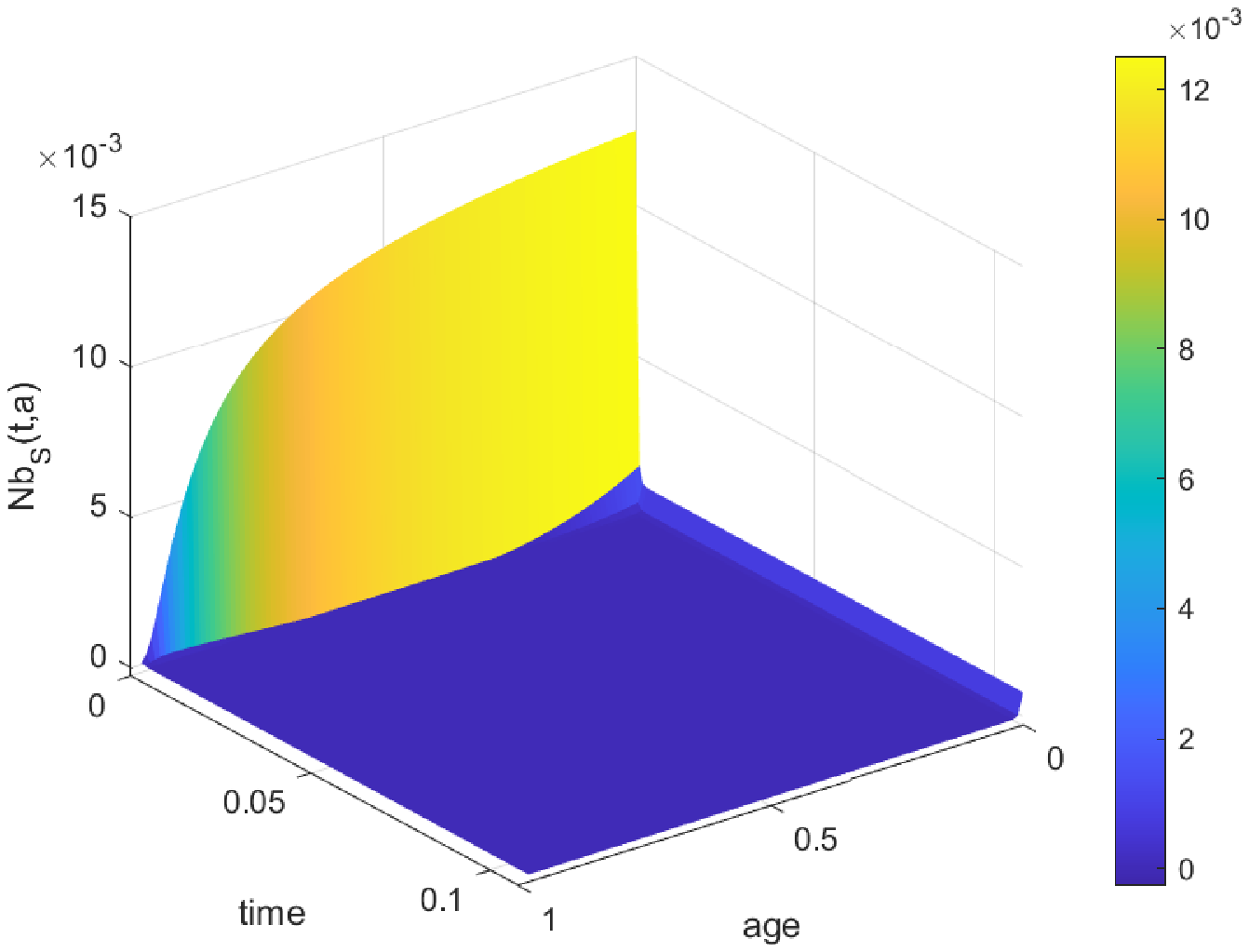}  
			\caption{Dynamics of S-individuals from PIDE Model with vaccination }  % width is 8.4 cm.
			\label{fig_s_pde}                                 % Size the figures 
		\end{center}                               
	\end{figure}
	\begin{figure}
		\begin{center}
			\includegraphics[height=5cm]{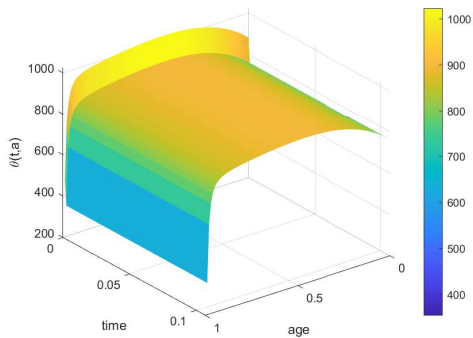}    % The printed column  
			\caption{Dynamics of vaccination for the PIDE Model with vaccination law \eqref{Theta}}  % width is 8.4 cm.
			\label{fig_vac_pde}                          
		\end{center}                                 
	\end{figure}
	\noindent Figure \ref{fig_i_pde} to \ref{fig_vac_pde} confirms theoretical results. Indeed, Figure \ref{fig_i_pde} shows that the I-individuals tends to zero as time increases. In Figure \ref{fig_s_pde}, the S-individuals trajectory remains positive, so as the vaccination that can be viewed in Figure \ref{fig_vac_pde}. This vaccination law differs from the one obtained for the ODE model. This can be explained by the large choice of design parameters for both models. Since those parameters are not chosen in the same way (randomly for the ODE case and to ensure positivity of the vaccination law in the PDE case) differences occur. The shape of Figure \ref{fig_vac_pde} suggests of vaccinating strongly individuals at the beginning of the epidemy with less focus on young and old inviduals.
	\section{Conclusion and Perspectives}
	The dynamical analysis of an age-dependent SIR model was performed, where we emphasized that the principle of linearized stability is applicable. It was done in Theorem \ref{linearised_stab} by using recent theory. 
	Then, two methods were used to positively stabilize an age-dependent SIR model. The first one is based on the discretization of the PIDE SIR Model according to the age. Then a linearizing nonlinear feedback law was found in Lemma \ref{lemma_fdb} for a system obtained via a change of variables. We proved in Theorem \ref{stab_inf} that this feedback ensures stability of the infected population for some well-chosen gains. Moreover, conditions to get positivity of trajectories were established in Theorem \ref{pos_cdt}. The second method followed from the previous one by using a formal limit. This led to a linearizing nonlinear feedback law for the PIDE Model \ref{PDE_SIR}. Conditions ensuring stability of the closed-loop system were obtained in Theorem \ref{stab_inf_PDE}.
	Finally, numerical simulations have corroborated theoretical results obtained with each methods.  \\ \\
	Some questions remain open. First we can notice that we did not impose a priori any condition on the positivity of the vaccination law, which is essential to have physical meaning. In numerical simulations a saturated law was used and seems to perform well. This could be theoretically validated. Moreover, currently, in the numerical simulations, the feedback gains are chosen randomly in order to satisfy the positivity and stability conditions. Another question could be the choice of those feedback gains in an optimal way. Finally, the control law that was designed is not applicable in practice since it requires the knowledge of all the state variables as it is a state feedback law. This is rarely the case in real situations. 
	A way to counter this is to use a state observer to estimate the whole state. The design of such an observer and the analysis of its performance in connection with the state feedback laws derived in this paper is an important question for further research.\\ \\
	\textbf{Acknowledgements}\\
	The first author wishes to thank the GIPSA-Lab (Grenoble, France) for the fruitful stay made in the framework of this research. This stay was made mostly possible by the support from the FRS-FNRS. The authors also wish to thank Christophe Prieur (GIPSA-Lab) for its insightful advices leading to significant improvements of the paper.
	
	\bibliographystyle{plain}        % Include this if you use bibtex 
	\bibliography{biblio}           % and a bib file to produce the 
	% bibliography (preferred). The
	% correct style is generated by
	% Elsevier at the time of printing.

	\appendix
	\section{Notation summary}\label{Appendix1}
	This section aims to summarize in Table \ref{notation_PDE} the parameters and variables used in this article for the PIDE models and for the ODE model respectively, their meaning and their domains of variation corresponding to their physical meaning. 
	
	\begin{table*}[b!]
		\centering
		\begin{tabular}{l|l|l|l}
			\hline
			\multicolumn{4}{l}{\textbf{PIDE Models}} \\ \hline
			Independent Variables	& Interpretation & Range & Unit \\ \hline
			$t$& Time    & $\mathbb{R}^+$ & day \\
			$L$& Maximum age & $\mathbb{R}^+$ & year \\
			$a$& Age  &  $\left[0,L\right]$ & year \\ 
			$\alpha$ & balancing coefficient  & $\mathbb{R}^+$& $\dfrac{\text{year}}{\text{day}}$\\ \hline
			Variables & & & \\ \hline
			& & & \vspace{-0.4cm}\\
			$P\left(t,a\right)$ &  Age density of the total population & $\mathbb{R}^+$ & $\dfrac{\text{Human}}{\text{day}}$\\
			\begin{tabular}[c]{@{}l@{}}$S\left(t,a\right)$\\[-0.15cm]$I\left(t,a\right)$\\[-0.15cm]$R\left(t,a\right)$\end{tabular}& \begin{tabular}[c]{@{}l@{}}Density of S-, I- and R-individuals at time $t$ and age $a$\end{tabular}   & $\mathbb{R}^+$ & $\dfrac{\text{Human}}{\text{day}}$ \\
			\begin{tabular}[c]{@{}l@{}}$s\left(t,a\right)$\\[-0.15cm]$i\left(t,a\right)$\\[-0.15cm]$r\left(t,a\right)$\end{tabular}& \begin{tabular}[c]{@{}l@{}}Normalized density of S-, I- and R-individuals at time $t$ and age $a$\end{tabular}  &  $\left[0,1\right]$ & no unit \\
			$\hat{s}\left(t,a\right)$& \begin{tabular}[c]{@{}l@{}}Normalized density of  S-individuals at time $t$ and age $a$\end{tabular}  &  $\left[-1,0\right]$ & no unit \\
			$\Theta\left(t,a\right)$ & \begin{tabular}[c]{@{}l@{}}Rate of vaccinated S-individuals\end{tabular}&  $\mathbb{R}^+$ & $\dfrac{ 1}{\text{day}}$  \\\hline 
			Parameters & & & \\\hline
			& & & \vspace{-0.4cm}\\
			$B$& birth rate   & $\mathbb{R}_0^+$ & $\dfrac{ \text{Human}}{\text{day}}$ \\
			$\mu\left(a\right)$& Per capita death rate   & $\mathbb{R}^+$ & $\dfrac{ 1}{\text{day}}$ \\
			$\beta\left(a\right)$ &\begin{tabular}[c]{@{}l@{}} Transmission coefficient between all I- and S-individuals of age $a$ \end{tabular} & $\mathbb{R}^+$ & 
			$\dfrac{ 1}{\text{Human}.\text{day}}$	\\	$\gamma\left(a\right)$ & Recovery rate  & $\mathbb{R}^+$ & $\dfrac{ 1}{\text{day}}$\\	
			
			\hline
			\multicolumn{4}{l}{\textbf{ODE Model}} \\ \hline
			
			Independent Variables & Interpretation & Range & Unit \\ \hline
			$t$& Time    & $\mathbb{R}^+$ &day \\ \hline
			Variables & & & \\ \hline
			$N_k$ &  \begin{tabular}[c]{@{}l@{}}Total number of individuals at age in $\left[a_{k-1},a_k\right[$ \end{tabular} & $\mathbb{R}^+$ & Human\\
			\begin{tabular}[c]{@{}l@{}}$s_k\left(t\right)$\\[-0.15cm]$i_k\left(t\right)$\\[-0.15cm]$r_k\left(t\right)$\end{tabular}& \begin{tabular}[c]{@{}l@{}}Proportion of  S-, I- and R-individuals at age in $\left[a_{k-1},a_k\right[$\end{tabular}  &  $\left[0,1\right]$ & no unit\\ 
			$\theta_k\left(t\right)$ & \begin{tabular}[c]{@{}l@{}}Rate of vaccinated S-individuals at age in $\left[a_{k-1},a_k\right[$  \end{tabular}&  $\mathbb{R}^+$ & $\dfrac{1}{day}$ \\\hline
			Parameters & & &\\ \hline
			& & & \vspace{-0.4cm}\\
			$\mu_k$& \begin{tabular}[c]{@{}l@{}}Per capita death rate at age in $\left[a_{k-1},a_k\right[$ \end{tabular}   & $\mathbb{R}^+$ & $\dfrac{1}{\text{day}}$ \\
			$\beta_k$ &\begin{tabular}[c]{@{}l@{}} Transmission coefficient  between all I- and S-individuals at age in $\left[a_{k-1},a_k\right[$  \end{tabular} & $\mathbb{R}^+$&$\dfrac{1}{\text{Hum.day}}$\\
			$\gamma_k$ & \begin{tabular}[c]{@{}l@{}} Recovery rate at age in $\left[a_{k-1},a_k\right[$ \end{tabular}  & $\mathbb{R}^+$ &$\dfrac{1}{\text{day}}$\\
			$\rho_k$ & \begin{tabular}[c]{@{}l@{}} Transfer rate from the $k$th class   \\[-0.15cm] of age to the $(k+1)$th\end{tabular} & $\mathbb{R}^+$ &$\dfrac{1}{\text{day}}$
			
		\end{tabular}
		\caption{Parameters and variables for PIDE and ODE models}
		\label{notation_PDE}		
	\end{table*}     
	\section{Boundary Control Systems on $L^1$ Spaces}\label{BC_L1}   
	First note that none of the arguments developed in \cite{tucsnak_weiss_2009} requires a scalar product which is only defined on Hilbert space, except for the computation of the adjoint operator. However, in $L^1$ spaces, this operator can be compute using duality bracket, which is the approach used here. \\
	Using theory in \cite{tucsnak_weiss_2009}, we need to find the operator $B$ such that the system $\dot{z}(t)=Lz(t)$ with $Gz(t)=z(0)=u(t)$, with $z \in \mathcal{D}\left(L\right)$ is equivalent to the system $\dot{z}(t)=\mathcal{A}z(t)+Bu(t)$. This operator satisfies, for $z\in Z$ and $\psi \in \mathcal{D}\left(\mathcal{A}^{\star}\right)$, \begin{equation}
		\left[Lz,\psi\right]-\left[z,\mathcal{A}^{\star}\psi\right]=\left[Gz,B^{\star}\psi\right].\label{B}
	\end{equation} 
	In order to apply this theory, System \eqref{SSF} can rewrite as $\dot{z}(t)=Lz(t)$ with $Gz(t)=z(0)$ where $z=\left(\bar{I},\bar{S}\right)^T$, $L=\bar{\mathcal{A}}: Z:=W^{1,1}\left(0,L\right)\times W^{1,1}\left(0,L\right) \to X:=L^1\left(0,L\right)\times L^1\left(0,L\right)$ and $G:Z\to U:=\mathbb{R}^2$. Therefore, we define $\mathcal{A}$ as $L$ restricted to $X_1:= \text{Ker } G$. Therefore, $\mathcal{A}=\bar{\mathcal{A}}_0:\mathcal{D}\left(\bar{\mathcal{A}}_0\right)\to X$. It remains to find the operator $B : U \to X_{-1}$ by using relation \eqref{B}.\\ \\
	First we can show that the adjoint of $\bar{\mathcal{A}}_0$ with $\mathcal{D}\left(\bar{\mathcal{A}}_0\right)$  is given by $$\bar{\mathcal{A}}_0^{\star}=\begin{pmatrix}
		\dfrac{d\cdot}{da} & G(a)I\\ I & \dfrac{d\cdot}{da}
	\end{pmatrix}$$ with $\mathcal{D}\left(\bar{\mathcal{A}}_0^{\star}\right)=\left\{y\in L^{\infty}\left(0,L\right)\times L^{\infty}\left(0,L\right) \text{such that}\right.\\\left. \left[\mathcal{A}\cdot, y\right]:L^1\left(0,L\right)\times L^1\left(0,L\right)\to \mathbb{K} \text{ is bounded and linear} \right.\\ \left. \text{ and }y_1\left(L\right)=y_2\left(L\right)=0\right\}. $
	The proof is based on the Riesz representation theorem which implies that $\left[\mathcal{A}z,y\right]=\displaystyle\int_X \mathcal{A}z . y\hspace{0.1cm} d\lambda$ for $z \in \mathcal{D}\left(\mathcal{A}\right)$ and $y\in\mathcal{D}\left(\mathcal{A}^{\star}\right)$ for $\lambda$ $ \sigma-$finite.  
	Then, using relation \eqref{B} and the Riesz representation theorem we find that \begin{align*}
		z\left(0\right)\odot B^{\star}\psi &=z(0)\odot \psi(0).\\
		\intertext{It follows that}
		B^{\star}\psi &=\psi(0)\\
		\left[\psi,B^{\star}\right]&=\left[\delta_0 I,\psi\right]\\&=\left[B,\psi\right]
	\end{align*}
	Therefore, $B$ is given by $\delta_0 I$. \\
	Finally, the assumptions needed in \cite{tucsnak_weiss_2009} are satisfied mostly since $\bar{\mathcal{A}}_0$ is the infinitesimal generator of a $C_0-$semigroup.

\end{document}